\documentclass[12pt]{aptpub}

\renewcommand{\crnotice}{\hfill{{\small\em

\/}\ {\small}}}

\usepackage{amsmath,amstext,url}

\renewcommand{\authornames}{\underline{\footnotesize }
L.Y. Li, J.P. Li.
}
\renewcommand{\shorttitle}
{ \underline{{Large deviation rates for supercritical multitype branching processes with immigration}}}

\numberwithin{equation}{section} 

\makeatletter \@addtoreset{equation}{section}

\makeatletter \@addtoreset{lemma}{section}

\makeatletter \@addtoreset{theorem}{section}

\makeatletter \@addtoreset{proposition}{section}

\makeatletter \@addtoreset{corollary}{section}

\makeatletter \@addtoreset{remark}{section}

\makeatletter \@addtoreset{definition}{section}

\makeatletter \@addtoreset{example}{section}

\begin{document}

\thispagestyle{firstpg}

\vspace*{1.5pc} \noindent \normalsize\textbf{\Large {Large deviation rates for supercritical multitype branching processes with immigration}} \hfill

\vspace{12pt}
\hspace*{0.75pc}{\small\textrm{\uppercase{Liuyan Li}}}
\hspace{-2pt}$^{*}$, {\small\textit{Guangdong University of Finance and Economics; Central
South University}}

\hspace*{0.75pc}{\small\textrm{\uppercase{Junping Li
}}}\hspace{-2pt}$^{**}$, {\small\textit{Guangdong University of Science and Technology; Central South University }}


\par
\footnote{\hspace*{-0.75pc}$^{*}\,$Postal
address: School of Statistics and Data Science, Guangdong University of Finance and Economics, Guangzhou, 510320, China. E-mail address:
mathlily@gdufe.edu.cn}

\par
\footnote{\hspace*{-0.75pc}$^{**}\,$Postal address:
Guangdong university of science and technology, Dongguan, 523083; School of Mathematics and Statistics, Central
South University, Changsha, 410083, China. E-mail address:
jpli@mail.csu.edu.cn }


\par
\renewenvironment{abstract}{%
\vspace{8pt} \vspace{0.1pc} \hspace*{0.25pc}
\begin{minipage}{14cm}
\footnotesize
{\bf Abstract}\\[1ex]
\hspace*{0.5pc}} {\end{minipage}}
\begin{abstract}
 Let $\{X_n\}_{n\geq0}$ be a $p$-type ($p\geq2$) supercritical branching process with immigration and mean matrix $M$. Suppose that $M$ is positively regular and $\rho$ is the maximal eigenvalue of $M$ with the corresponding left and right eigenvectors $\boldsymbol{v}$ and $\boldsymbol{u}$. Let $\rho>1$ and $Y_n=\rho^{-n}\Big[\boldsymbol{u}\cdot X_n -\frac{\rho^{n+1}-1}{\rho-1}( \boldsymbol{u}\cdot \boldsymbol{\lambda})\Big]$, where the vector $\boldsymbol{\lambda}$ denotes the mean immigration rate. In this paper, we will show that $Y_n$ is a martingale and converges to a $r.v.$ $Y$ as $n\rightarrow\infty$. We study the rates of convergence to $0$ as $n\rightarrow\infty$ of
 $$
 P_i\Big(\Big|\frac{\boldsymbol{l}\cdot X_{n+1}}{\textbf{1}\cdot X_n}-\frac{\boldsymbol{l}\cdot(X_nM)}{\textbf{1}\cdot X_n}\Big|>\varepsilon\Big),P_i\Big(\Big|\frac{\boldsymbol{l}\cdot X_n}{\textbf{1}\cdot X_n}-\frac{\boldsymbol{l}\cdot\boldsymbol{v}}{\textbf{1}\cdot \boldsymbol{v} }\Big|>\varepsilon\Big),P\Big(\Big|Y_n-Y\Big|>\varepsilon\Big)
 $$
  for any $\varepsilon>0, i=1,\cdots,p$, $\textbf{1}=(1,\cdots,1)$ and $\boldsymbol{l}\in\mathbb{R}^p,$ the $p$-dimensional Euclidean space. It is shown that under certain moment conditions, the first two decay geometrically, while conditionally on the event $Y\geq\alpha$ $(\alpha>0)$ supergeometrically. The decay rate of the last probability is always supergeometric under a finite moment generating function assumption.
\end{abstract}

\vspace*{12pt} \hspace*{2.25pc}
\parbox[b]{26.75pc}{{
}}
\par
\normalsize

\renewcommand{\amsprimary}[1]{
     \vspace*{8pt}
     \hspace*{2.25pc}
     \parbox[b]{24.75pc}{\scriptsize
    AMS 2000 Subject Classification: Primary 60J27 Secondary 60J35
     {\uppercase{#1}}}\par\normalsize}
\renewcommand{\ams}[2]{
     \vspace*{8pt}
     \hspace*{2.25pc}
     \parbox[b]{24.75pc}{\scriptsize
     AMS 2020 SUBJECT CLASSIFICATION: PRIMARY
     {\uppercase{#1}}\\ \phantom{
     AMS 2020
     SUBJECT CLASSIFICATION:
     }
    SECONDARY
 {\uppercase{#2}}}\par\normalsize}

\ams{60J27}{60J35}

\par
\vspace{5mm}
 \setcounter{section}{1}
 \setcounter{equation}{0}
 \setcounter{theorem}{0}
 \setcounter{lemma}{0}
 \setcounter{corollary}{0}
\noindent {\large \bf 1. Introduction}
\vspace{3mm}
\par
 Let $\big\{Z_n=\big(Z_n^{(1)},\cdots,Z_n^{(p)}\big)\big\}_{n\geq 0}$ be a $p$-type $(p\geq2)$ supercritical branching process
   with mean matrix $M$, where $Z_n^{(r)}$$(1\leq r\leq p)$ denotes the number of $r$'th type particles in the $n$'th generation.
    Suppose that there is an immigration of $I_n=\big(I_n^{(1)},\cdots,I_n^{(p)}\big)$ particles at time $n$ into the population,
    where $I_n^{(i)}$$(1\leq i\leq p)$ is the number of $i$'th type particles in the $n$'th immigration. By Kaplan \cite{KAP74},
    a $p$-type supercritical branching process with immigration $\{X_n\}_{n\geq0}$ can be defined as
   \begin{equation}\label{1.1}
   X_n=Z_n+U_n^{(1)}+U_n^{(2)}+\cdots+U_n^{(n)},\ \ n\ge1,
  \end{equation}
   for fixed $1\leq k\leq n,$ random vector $U_n^{(k)}=\big(U_{n,1}^{(k)},\cdots,U_{n,p}^{(k)}\big),$ where the $j$'th component $U_{n,j}^{(k)}$
   is the number of $j$'th type descendants in the $n$'th generation from the $k$'th immigration $I_k$. We note that $U_n^{(n)}=I_n$
   and all the particles in the system are independent. The new particles behave as a $p$-type branching process. For convenience, we consider $p=2$ in this paper.
\par
Let $\mathbb{N}$ be the set of natural numbers and $\mathbb{N}^2=\{\emph{\textbf{j}}=(j_1,j_2):j_1\in \mathbb{N},j_2\in\mathbb{N}\}$. For
   $\emph{\textbf{s}}\in [0,1]^2$ and $\emph{\textbf{j}}\in \mathbb{N}^2,$
   ${\emph{\textbf{s}}}^{\emph{\textbf{j}}}=s_1^{j_1}s_2^{j_2}.$
   Denote $\textbf{1}=(1,1),$ $\textbf{0}=(0,0),$ $e_1=(1,0)$ and  $e_2=(0,1).$  Let
  $f_1^{(i)}(\emph{\textbf{s}})$ be the generating function of $Z_1$ with $Z_0=e_i$ for $i=1,2,$ then
 $$
  f_1^{(i)}(\emph{\textbf{s}})=E[\emph{\textbf{s}}^{Z_1}|Z_0=e_i]
  =\sum_{\emph{\textbf{j}}\in \mathbb{N}^2} P_i(j_1,j_2)s_1^{j_1}s_2^{j_2},
 $$
 where $P_i(j_1,j_2)=P\big(Z_1^{(1)}=j_1,Z_1^{(2)}=j_2|Z_0=e_i\big).$
 Let $f_n^{(i)}(\emph{\textbf{s}})$ be the generating function of $Z_n$ with $Z_0=e_i$ for $i=1,2.$  Then we can write
  $$
  f_n(\emph{\textbf{s}})= \big(f_n^{(1)}(\emph{\textbf{s}}),f_n^{(2)}(\emph{\textbf{s}})\big).
  $$
 In particular, $f_0(\emph{\textbf{s}})\equiv \emph{\textbf{s}}, f(\emph{\textbf{s}})\equiv f_1(\emph{\textbf{s}}).$
 It is known that (see \cite{AKN72}) $f_{n+1}(\emph{\textbf{s}})=f(f_n(\emph{\textbf{s}}))$.
 \par
 Assume that $\{I_n\}_{n\geq1}$ are $i.i.d.$ $2$-dimensional random vectors and have generating function $h(\emph{\textbf{s}}):=\sum\limits_{\emph{\textbf{j}}\in \mathbb{N}^2} h_{\emph{\textbf{j}}}s_1^{j_1}s_2^{j_2}$.
 Let $E(I_n)=\boldsymbol{\lambda}=(\lambda_1,\lambda_2)$ and $h_0=P(I_n=\bf{0}).$
   Let $g_n^{(i)}(\emph{\textbf{s}})$ be the generating function of $X_n$ with $X_0=e_i,$ then by $(1.1)$ we see that
$$
g_n^{(i)}(\emph{\textbf{s}})
=E\big[\emph{\textbf{s}}^{X_n}|X_0=e_i\big]
=f_n^{(i)}(\emph{\textbf{s}})\prod_{k=1}^{n}h(f_{n-k}(\emph{\textbf{s}})),
$$
 where $h(f_{n-k}(\emph{\textbf{s}}))$ is the generating function of $U_n^{(k)}.$
We can also write
 \begin{equation}\label{1.2}
 g_n(\emph{\textbf{s}})=\big(g_n^{(1)}(\emph{\textbf{s}}),g_n^{(2)}(\emph{\textbf{s}})\big)
 =f_n(\emph{\textbf{s}})\prod_{k=0}^{n-1}h(f_k(\emph{\textbf{s}})).
 \end{equation}
 In particular, $g_0(\emph{\textbf{s}})\equiv \emph{\textbf{s}}$, $g(\emph{\textbf{s}})\equiv g_1(\emph{\textbf{s}}).$ Define
 $D_{ij}(\emph{\textbf{s}})=\frac{\partial f_1^{(i)}(\emph{\textbf{s}})}{\partial s_j}$ for $i,j=1,2$
  and $\emph{\textbf{s}}\in [0,1]^2.$ Denote $a_{ij}=D_{ij}(\textbf{0}), m_{ij}=D_{ij}(\textbf{1}^{-})$, $A=(a_{ij})$. $M=(m_{ij})$ is called the mean matrix.
\par
    Throughout this paper, we suppose that $f(\textbf{0})=\textbf{0}$ (meaning that every particle will not die) and $M$ is positively regular which ensures the existence of maximum eigenvalue $\rho$ of $M$ (see \cite{JAP67} or \cite{KAR66}).
 We assume that $\rho>1$ (in this case that the extinction probability of $\{Z_n\}$ is less than $1$) and $\boldsymbol{v}$ and $\boldsymbol{u}$ are the corresponding left and right eigenvector. It is known from \cite{JAP67} or \cite{KAR66} that all components of $\boldsymbol{v}$ and $\boldsymbol{u}$ are positive, $\boldsymbol{v}\cdot\boldsymbol{u}=1$ and $\boldsymbol{u}\cdot\boldsymbol{1}=1$, where the normalization being assumed for the sake of convenience.
  Let $P_i(\cdot)=P(\cdot|X_0=e_i)$ be the probability measure
 for the process conditioned on $X_0=e_i$ and $E_i(\cdot)=E(\cdot|X_0=e_i)$ be the corresponding expectation for $i=1,2.$
 \par
  Let $\mathbb{R}^2$ be the $2$-dimensional Euclidean space. It is known from \cite{AV95} that for any $\boldsymbol{l}$ in $\mathbb{R}^2,$
 $$
 \frac{\boldsymbol{l}\cdot Z_{n+1}}{\textbf{1}\cdot Z_n}-\frac{\boldsymbol{l}\cdot(Z_nM)}{\textbf{1}\cdot Z_n}\ \ \ \ \text{and} \ \ \ \
  \frac{\boldsymbol{l}\cdot Z_n}{\textbf{1}\cdot Z_n}-\frac{\boldsymbol{l}\cdot\boldsymbol{v}}{\boldsymbol{1}\cdot \boldsymbol{v} }
 $$
 converge to 0 with probability 1 (w.p.1) on the set of nonextinction. Athreya and Vidyashan-kar \cite{AV95} studied
 the large deviation aspects of the above convergence. This paper extends the convergence and large deviation
 results in \cite{AV95} to supercritical multitype branching processes with immigration.
 \par
 In this paper, we will show that $Y_n=\rho^{-n}\big(\boldsymbol{u}\cdot X_n-\frac{\rho^{n+1}-1}{\rho-1}(\boldsymbol{u}\cdot \boldsymbol{\lambda})\big)$
 is a martingale and converges to a $r.v.$ $Y$ w.p.1.
 Then we will study the rates of convergence to $0$ as $n\to\infty$ of
  $$
  P_i\Big(\Big|\frac{\boldsymbol{l}\cdot X_{n+1}}{\textbf{1}\cdot X_n}-\frac{\boldsymbol{l}\cdot(X_nM)}{\textbf{1}\cdot X_n}\Big|>\varepsilon\Big)
  ,\ \ \ P_i\Big(\Big|\frac{\boldsymbol{l}\cdot X_n}{\textbf{1}\cdot X_n}-\frac{\boldsymbol{l}\cdot \boldsymbol{v} }{\textbf{1}\cdot \boldsymbol{v} }\Big|>\varepsilon\Big),
   \ \ \ P(|Y_n-Y|>\varepsilon)
  $$
  $$
  P_i\Big(\Big|\frac{\boldsymbol{l}\cdot X_{n+1}}{\textbf{1}\cdot X_n}-\frac{\boldsymbol{l}\cdot(X_nM)}
  {\textbf{1}\cdot X_n}\Big|>\varepsilon\Big|Y\geq\alpha\Big),\ \ \ \
  P_i\Big(\Big|\frac{\boldsymbol{l}\cdot X_n}{\textbf{1}\cdot X_n}-\frac{\boldsymbol{l}\cdot \boldsymbol{v} }{\textbf{1}\cdot \boldsymbol{v} }\Big|>\varepsilon\Big|Y\geq\alpha\Big)
  $$
  for any $\varepsilon>0, \boldsymbol{l}\in\mathbb{R}^2, i=1,2$ and $\alpha>0.$
  \par
  Investigating the rate of convergence of these probabilities is not only interesting in its own right, but also arises in computer science applications (see Karp and Zhang \cite{KZ94}). Athreya \cite{Ath94} studied
    the large deviations for single-type Galton-Watson branching processes. On the basis of Athreya \cite{Ath94},
     Li L. and Li J. \cite{LL19} obtained the decay rates corresponding to the above third and fourth large deviation probabilities
     for a single-type Galton-Watson process with immigration. Later, Li J, et al. \cite{LCCL19,LCL19} investigated
      the above large deviation probabilities for single-type continuous time branching processes.
Liu and Zhang \cite{LZ16} studied the decay rate of the first probability for a single-type Galton-Watson process with immigration.
\par
\vspace{5mm}
 \setcounter{section}{1}
 \setcounter{equation}{0}
 \setcounter{theorem}{0}
 \setcounter{lemma}{0}
 \setcounter{corollary}{0}
\noindent {\large \bf 2. Main results}
\vspace{3mm}
\par
\par
\par
The first theorem is a large deviation result of the process under an exponential moment hypothesis on the offspring and immigration distribution functions. It can be obtained that the two probabilities in Theorem~\ref{th3.1} below decay to $0$ with geometrical rates. In the non-immigration and single-type cases, the geometrical decay rate was proved by Athreya (\cite{Ath94, AV95}) under an exponential moment hypothesis on the offspring distribution function.

\begin{theorem}\label{th3.1} Assume that $h_0>0$ and there is $0<\gamma<1$ such that $A^n\gamma^{-n}$ converges
to a matrix $P_0$ which is nonzero and has finite entries. Suppose that
 $ E_i(e^{\theta_0(\boldsymbol{1}\cdot X_1)})<\infty$
for some $\theta_0>0$ and $i=1,2.$ Let $\boldsymbol{l}=(l_1, l_2)$ be a nonzero vector with $l_1\neq l_2.$
 Then, for any $\varepsilon>0$ and $i=1,2,$
\begin{equation}\label{3.1}
\lim_{n\rightarrow\infty}(h_0\gamma)^{-n}P_i\Big(\Big|\frac{\boldsymbol{l}\cdot X_{n+1}}{\emph{\textbf{1}}\cdot X_n}
-\frac{\boldsymbol{l}\cdot(X_nM)}{\emph{\textbf{1}}\cdot X_n}\Big|>\varepsilon\Big)
=\sum_{\boldsymbol{j}\in\mathbb{N}^2\setminus\{\boldsymbol{0}\}}\phi(\boldsymbol{j},\varepsilon)r_{i,\boldsymbol{j}}<\infty,
\end{equation}
 and there is positive integer $k_0$ such that
\begin{equation}\label{3.2}
\lim_{n\rightarrow\infty}(h_0\gamma)^{-n}P_i\Big(\Big|\frac{\boldsymbol{l}\cdot X_n}{\emph{\textbf{1}}\cdot X_n}
-\frac{\boldsymbol{l}\cdot \boldsymbol{v} }{\emph{\textbf{1}}\cdot \boldsymbol{v} }\Big|>\varepsilon\Big)
=(h_0\gamma)^{k_0}\sum_{\boldsymbol{j}\in\mathbb{N}^2\setminus\{\boldsymbol{0}\}}
\hat{\phi}(\boldsymbol{j},k_0,\varepsilon)r_{i,\boldsymbol{j}}<\infty,
\end{equation}
 where
$$
\phi(\boldsymbol{j},\varepsilon)=P(|\boldsymbol{l}\cdot X_{1}-\boldsymbol{l}\cdot(\boldsymbol{j}M)|
>\varepsilon(\textbf{1}\cdot \boldsymbol{j})|X_0=\boldsymbol{j}),
$$
$r_{i,\boldsymbol{j}}$ is given by \eqref{2.4a} below and
$$
\hat{\phi}(\boldsymbol{j},k_0,\varepsilon)= P\Big(\Big|\frac{\boldsymbol{l}\cdot X_{k_0}}{\textbf{1}\cdot X_{k_0}}-\frac{\boldsymbol{l}\cdot \boldsymbol{v}}{\textbf{1}\cdot \boldsymbol{v} }\Big|>\varepsilon\Big|X_{0}=\boldsymbol{j}\Big).
$$
\end{theorem}
\par
Similar to what Athreya and Vidyashankar have showed in the non-immigration case \cite{AV95}, the exponential condition in Theorem \ref{th3.1} can be weakened to a polynomial condition.
\begin{theorem}\label{th3.2} Assume that $h_0>0$ and there is $0<\gamma<1$ such that $A^n\gamma^{-n}$ converges
 to a matrix $P_0$ which is nonzero and has finite entries. Assume that
  $E_i\big((\emph{\textbf{1}}\cdot Z_1)^{2s}\big)<\infty$, $E\big((\emph{\textbf{1}}\cdot I_1)^{\kappa}\big)<\infty$ for some $s\geq1, \kappa\geq1, i=1,2$ and $h_0\gamma\rho^{d}>1$, where $d=\min\{s,\kappa\}$. Let $\boldsymbol{l}=(l_1, l_2)$ be a nonzero vector
 with $l_1\neq l_2.$ Then, for any $\varepsilon>0$ and $i=1,2,$ (\ref{3.1}) and (\ref{3.2}) also hold.
\end{theorem}

\begin{remark}\label{re:1.2}
In particular, when $s=1$, $\kappa\geq 1$ and $h_0\gamma\rho>1$, the conditions for moments in Theorem \ref{th3.2} is obviously satisfied. When $s>1$, as long as $\kappa\geq 1$ and $h_0\gamma\rho^d>1$, the conditions of Theorem \ref{th3.2} regarding moments are satisfied, while in this case, $\kappa$ may be greater than or equal to $s$, or it may be less than $s$.
However, Theorem \ref{th3.2} may not hold when the immigration distribution and offspring distribution both have finite $\kappa$-order moments, for example, in the case $1\leq\kappa<2$, $E_i\big((\boldsymbol{1}\cdot Z_1)^{2}\big)$ could be infinite.
\end{remark}

 The following result obtains that the decay rates of the two probabilities in Theorem~\ref{th3.1} are supergeometric, and it extends the existing results in Athreya and Vidyashankar \cite{AV95} for the non-immigration case.

\begin{theorem}\label{th3.3}
Assume that $ E_i\big(e^{\theta_0(\emph{\textbf{1}}\cdot X_1)}\big)<\infty$
for some $\theta_0>0$ and $i=1,2.$ Suppose that
\begin{equation}\label{2.7}
\sum_{i=2}^{\infty}P_1(i,0)>0,\ \sum_{j=2}^{\infty}P_2(0,j)>0,\ \sum_{i=1}^{\infty}P(I_1=(i,0))>0\
,\sum_{j=1}^{\infty}P(I_1=(0,j))>0
\end{equation}
and
\begin{equation}\label{2.8}
P_i\big(Z_1^{(i)}<k_i\big)=0, P\big(I_1^{(i)}<d_i\big)=0 \ \ for\ i=1,2,
\end{equation}
where $k_1=\inf\{i\geq2; P_1(i,0)>0\},$ $k_2=\inf\{j\geq2; P_2(0,j)>0\}$ and $d_1=\inf\{i\geq1;
P(I_1=(i,0)>0)\},$ $d_2=\inf\{ j\geq1; P(I_1=(0,j)>0)\}.$
Then, for any nonzero vector $\boldsymbol{l}$ and $\varepsilon>0,$ there are constants $0<C_1(\varepsilon),C_2(\varepsilon)<\infty$
and $0<\mu_1(\varepsilon),\mu_2(\varepsilon)<1$ such that for $i=1,2$,
\begin{equation}\label{3.13}
  P_i\Big(\Big|\frac{\boldsymbol{l}\cdot X_{n+1}}{\emph{\textbf{1}}\cdot X_n}-\frac{\boldsymbol{l}\cdot (X_nM)}{\emph{\textbf{1}}\cdot X_n}
  \Big|>\varepsilon\Big)
  \leq C_1(\varepsilon)(\mu_1(\varepsilon))^{k_i^n}
\end{equation}
  and
\begin{equation}\label{3.14}
P_i\Big(\Big|\frac{\boldsymbol{l}\cdot X_n}{\emph{\textbf{1}}\cdot X_n}-\frac{\boldsymbol{l}\cdot \boldsymbol{v} }{\emph{\textbf{1}}\cdot \boldsymbol{v} }\Big|>\varepsilon\Big)
\leq C_2(\varepsilon)(\mu_2(\varepsilon))^{k_i^n}.
\end{equation}
\end{theorem}

\begin{remark}\label{re:1.3}
(\ref{2.7}) and (\ref{2.8}) mean that every $i$ type particle produces at least $k_i$ particles of it's kind and there are at least $d_i$ particles of type $i$ in each immigration. Athreya and Vidyashankar \cite{AV95} showed the case $k_i=2$ in the non-immigration case.
\end{remark}
\par
Recalling that
$Y_n=\rho^{-n}\big(\boldsymbol{u}\cdot X_n-\frac{\rho^{n+1}-1}{\rho-1}(\boldsymbol{u}\cdot \boldsymbol{\lambda})\big)$. We will show that $\lim\limits_{n\to\infty}Y_n=Y$ w.p.1 in Proposition~\ref{pro2.3}. The next theorem analyzes the properties of $Y_n$ and is needed in the proof of Theorem~\ref{th3.5}. A proof of the discrete time case can be found, e.g., in \cite{Ath94,AV95,LL19}, and for the continuous time case in \cite{LCCL19,LCL19}.

\begin{theorem}\label{th3.4}
 Assume that $ E_i\big(e^{\theta_0(\emph{\textbf{1}}\cdot X_1)}\big)<\infty$
for some $\theta_0>0$ and $i=1,2$, then there is $\theta_0^{\star}>0$ such that
\begin{equation}\label{3.15}
C_3=sup_n E_i\big(e^{\theta_0^{\star}Y_n}\big)<\infty.
\end{equation}
\end{theorem}
\par
Theorem~\ref{th3.5} below shows that the decay rate of $P(\left|Y_n-Y\right|>\varepsilon)$ is supergeometric, as we have seen in the non-immigration case (see~\cite{AV95,LCCL19}) and the single-type case
(see~\cite{Ath94,LL19}).
\par
\begin{theorem} \label{th3.5}
 Assume that $ E_i\big(e^{\theta_0(\emph{\textbf{1}}\cdot X_1)}\big)<\infty$
for some $\theta_0>0$ and $i=1,2$,
then there are constants $0<C_4<\infty$ and $0<\kappa_1<\infty$ such that for $\varepsilon>0,i=1,2,$
\begin{equation}\label{3.18}
P_i(\left|Y_n-Y\right|>\varepsilon)\leq C_4e^{-\kappa_1\varepsilon^\frac{2}{3}\big(\rho^\frac{1}{3}\big)^n}.
\end{equation}
\end{theorem}
\par
The following result shows that conditioned on $Y\geq \alpha$, $\alpha>0,$ the large deviation probabilities
in Theorem \ref{th3.1} decay super-geometrically. Similar results hold in the non-immigration case (see~\cite{AV95}) and the single-type
case (see~\cite{Ath94,LL19}).
\par
\begin{theorem} \label{th3.6}
  Assume that $ E_i(e^{\theta_0(\emph{\textbf{1}}\cdot X_1)})<\infty$
for some $\theta_0>0$ and $i=1,2$.
 Then, there are constants $0<C_5,C_6,C_7,C_8<\infty$ and $0<\kappa_2,\kappa_3<\infty$ such that for every
 $\varepsilon>0,$ $\alpha>0$ and $\boldsymbol{l}\neq\emph{\textbf{0}},$ we can find $0<\Lambda(\varepsilon)<\infty$ such that
\begin{equation}\label{3.21}
P_i\Big(\Big|\frac{\boldsymbol{l}\cdot X_{n+1}}{\emph{\textbf{1}}\cdot X_n}-\frac{\boldsymbol{l}\cdot (X_nM)}{\emph{\textbf{1}}\cdot X_n}\Big|
>\varepsilon\Big|Y\geq \alpha\Big)\leq C_5e^{-\alpha\beta \Lambda(\varepsilon)\rho^n}
+C_6e^{-\kappa_2\big(\alpha(1-\beta)\big)^\frac{2}{3}\big(\rho^\frac{1}{3}\big)^n}
\end{equation}
and
\begin{equation}\label{3.22}
P_i\Big(\Big|\frac{\boldsymbol{l}\cdot X_n}{\emph{\textbf{1}}\cdot X_n}-\frac{\boldsymbol{l}\cdot \boldsymbol{v} }{\emph{\textbf{1}}\cdot \boldsymbol{v} }\Big|
>\varepsilon\Big|Y\geq \alpha\Big)\leq C_7e^{-\alpha\beta \Lambda(\varepsilon)\rho^n}
+C_8e^{-\kappa_3\big(\alpha(1-\beta)\big)^\frac{2}{3}\big(\rho^\frac{1}{3}\big)^n}
\end{equation}
 for every $0<\beta<1$.
 \end{theorem}

\par
\vspace{5mm}
 \setcounter{section}{3}
 \setcounter{equation}{0}
 \setcounter{theorem}{0}
 \setcounter{lemma}{0}
 \setcounter{definition}{0}
 \setcounter{corollary}{0}
\noindent {\large \bf 3. Proofs}
 \vspace{3mm}
\par
We now turn to the proofs of the main results. To prove Theorem~\ref{th3.1}, we shall use the following two important results. The first one is a large deviation result for the process under an exponential moment hypothesis and shows that the decay rate of $g_n(\emph{\textbf{s}})$ is geometric.
Define $\Vert \emph{\textbf{a}}\Vert=max\{|a_1|, |a_2|\}$ for vector $\emph{\textbf{a}}\in\mathbb{R}^2$. Let $\mathbb{R}_{+}^2=\{(r_1,r_2):r_1\geq0,r_2\geq0\}$.
\par
\noindent
\begin{proposition}\label{pro2.1}
Assume that $h_0>0$ and there is $0<\gamma<1$ such that $A^n\gamma^{-n}$
converges to a matrix $P_0$ which is nonzero and has finite entries. Then, there exists
$R:[0,1]^2\rightarrow\mathbb{R}_{+}^2$ such that for every $\boldsymbol{s}\in[0,1]^2\setminus\{\boldsymbol{1}\}$,
\begin{equation}\label{2.1}
\frac{g_n(\textbf{s})}{(h_0\gamma)^n}\rightarrow R(\textbf{s}) \ \ \ as \ n\rightarrow\infty,
\end{equation}
and
$R(\textbf{s})$ is the unique solution to the vector functional equation below
\begin{equation}\label{2.2}
h(\textbf{s})R(f(\textbf{s}))=h_0\gamma R(\textbf{s}),\ \boldsymbol{s}\in[0,1]^2\setminus\{\boldsymbol{1}\},
\end{equation}
subject to
\begin{equation}\label{2.3}
 R(\emph{\textbf{0}})=\emph{\textbf{0}},\ \ \ \ 0<\Vert R(\boldsymbol{s})\Vert <\infty \ \
 for \ \boldsymbol{s}\in[0,1]^2\setminus\{\boldsymbol{0},\boldsymbol{1}\}\
\end{equation}
 and
\begin{equation}\label{2.4}
\lim_{\textbf{s}\rightarrow \emph{\textbf{1}}}R(\textbf{s})=\infty, \ \ \
\lim_{\textbf{s}\rightarrow \emph{\textbf{0}}}R'(\textbf{s})=P_0,
\end{equation}
where $R'$ is the Jacobian matrix. Furthermore, for $i=1,2$ and $\boldsymbol{j}\in \mathbb{N}^2\setminus\{\boldsymbol{0}\}$,
\begin{equation}\label{2.4a}
\lim_{n\to\infty}\frac{P_i(X_n=\boldsymbol{j})}{(h_0\gamma)^n}=r_{i,\boldsymbol{j}}\  \text{exists}
\end{equation}
and $R_i(\boldsymbol{s})=\sum\limits_{\boldsymbol{j}\in\mathbb{N}^2\setminus\{\boldsymbol{0}\}}
r_{i,\boldsymbol{j}}\boldsymbol{s}^{\boldsymbol{j}}$ for $\boldsymbol{s}\in[0,1]^2\setminus\{\boldsymbol{1}\}$, where $R(\boldsymbol{s})=(R_1(\boldsymbol{s}), R_2(\boldsymbol{s}))$.
\end{proposition}
\noindent
\begin{proof}
For $\boldsymbol{s}\in[0,1]^2\setminus\{\boldsymbol{1}\}$, it is known from Athreya and Vidyashankar \cite[Theorem 1]{AV95} that
$\widetilde{Q}_n(\emph{\textbf{s}})=\frac{f_n(\emph{\textbf{s}})}{\gamma^n}\rightarrow \widetilde{Q}(\emph{\textbf{s}})$ satisfying that $\widetilde{Q}(\boldsymbol{0})=\boldsymbol{0}$ and $0<\Vert \widetilde{Q}(\boldsymbol{s})\Vert <\infty$ for $\boldsymbol{s}\in[0,1]^2\setminus\{\boldsymbol{1}\}$.
Define
$S_n(\emph{\textbf{s}})
=h_0^{-n}\prod\limits_{k=0}^{n-1}h(f_k(\emph{\textbf{s}}))$ for
$\boldsymbol{s}\in[0,1]^2\setminus\{\boldsymbol{1}\}$.
Then
$S_n(\emph{\textbf{s}})$ is an increasing sequence for fixed $\emph{\textbf{s}}$.
From (\ref{1.2}), to prove (\ref{2.1}), it suffices to show that $\lim\limits_{n\rightarrow\infty}S_n(\emph{\textbf{s}})
=\prod\limits_{k=0}^{\infty}\frac{h(f_k(\emph{\textbf{s}}))}{h_0}=S(\emph{\textbf{s}})<\infty,$
which is equivalent to showing
\begin{equation}\label{2.5}
 \sum_{k=0}^{\infty}\big(h(f_k(\emph{\textbf{s}}))-h_0\big)<\infty.
\end{equation}
Applying the mean value theorem, for $\emph{\textbf{s}}\in [0,1]^2\setminus\{\textbf{1}\},$
there is a $0<\theta<1$ such that
$$
 h(f_k(\emph{\textbf{s}}))-h_0
 = \sum_{j=1}^2f_k^{(j)}(\emph{\textbf{s}}) \frac{\partial h(f_k(\theta \emph{\textbf{s}}))}{\partial s_j}
 \leq \Vert f_k(\emph{\textbf{s}})\Vert \sum_{j=1}^2\frac{\partial h(f_k(\theta \emph{\textbf{s}}))}{\partial s_j}
 \leq  C_9\Vert f_k(\emph{\textbf{s}})\Vert,
 $$
where the existence of constant $C_9$ is due to the continuity of $\sum\limits_{j=1}^2
\frac{\partial h(f_k(\theta \emph{\textbf{s}}))}{\partial s_j}.$
 By Athreya and Vidyashankar \cite[Lemma 2]{AV95}, there are constants $C_{10}$ and $0<\delta<1$
 (depending on $\emph{\textbf{s}}$) such that $\Vert f_k(\emph{\textbf{s}})\Vert\leq C_{10}\delta^k$,
 which implies (\ref{2.5}). From (\ref{1.2}) we observe that
\begin{equation}\label{2.6}
\frac{g_n(f(\emph{\textbf{s}}))}{(h_0\gamma)^n}
=\frac{f_{n+1}(\emph{\textbf{s}})\prod_{k=1}^n{h(f_k(\emph{\textbf{s}}))}}{(h_0\gamma)^n}.
\end{equation}
Multiplying both sides of (\ref{2.6}) by $h(\emph{\textbf{s}})$ 
  and letting $n\rightarrow \infty,$ we obtain (\ref{2.2}).
 \par
Recalling that $\widetilde{Q}(\boldsymbol{0})=\boldsymbol{0}$, $\Vert \widetilde{Q}(\boldsymbol{s})\Vert$ and $S(\boldsymbol{s})$ are positive and finite for $\boldsymbol{s}\in[0,1]^2\setminus\{\boldsymbol{1}\}$. This yields that $R(\boldsymbol{0})=\boldsymbol{0}$ and $0<\Vert R(\boldsymbol{s})\Vert <\infty$ for $\boldsymbol{s}\in[0,1]^2\setminus\{\boldsymbol{0},\boldsymbol{1}\}$. By (\ref{2.2}), we know that $R(\emph{\textbf{s}})\geq \frac{R(f(\emph{\textbf{s}}))}{\gamma}$ for any $\boldsymbol{s}\in[0,1]^2\setminus\{\boldsymbol{0},\boldsymbol{1}\}$. Recursively, we have $R(\emph{\textbf{s}})\geq \frac{R(f_n(\emph{\textbf{s}}))}{\gamma^n}$ for any $\boldsymbol{s}\in[0,1]^2\setminus\{\boldsymbol{0},\boldsymbol{1}\}$ and $n\geq 1$. We can find $\emph{\textbf{s}}_n\uparrow \boldsymbol{1}$ as $n\uparrow\infty$ such that $f_n(\emph{\textbf{s}}_n)\uparrow \boldsymbol{1}$. Since $\gamma^n\downarrow 0$ as $n\uparrow\infty$, we will have $R(\emph{\textbf{s}}_n)\uparrow\infty$.
 Thus, $\lim\limits_{\emph{\textbf{s}}\rightarrow \textbf{1}}R(\emph{\textbf{s}})=\infty$.
 By \cite[Theorem 1]{AV95}, we obtain that
  $$
  \lim_{\emph{\textbf{s}}\rightarrow \textbf{0}}R'(\emph{\textbf{s}})=\lim_{\emph{\textbf{s}}\rightarrow \textbf{0}}(\widetilde{Q}(\emph{\textbf{s}})S(\emph{\textbf{s}}))'=\widetilde{Q}'(\textbf{0})=P_0.
  $$
Let $R(\emph{\textbf{s}})$ and $\widetilde{R}(\emph{\textbf{s}})$ be any two solutions
 that satisfy (\ref{2.2})-(\ref{2.4}), then
\begin{eqnarray*}
\Vert R(\emph{\textbf{s}})-\widetilde{R}(\emph{\textbf{s}})\Vert
& =& (h_0\gamma)^{-n}\Big\Vert \prod_{k=0}^{n-1}h(f_k(\emph{\textbf{s}}))R(f_n(\emph{\textbf{s}}))
  -\prod_{k=0}^{n-1}h(f_k(\emph{\textbf{s}}))\widetilde{R}(f_n(\emph{\textbf{s}}))\Big\Vert \\
& \leq  & \frac{\Vert g_n(\emph{\textbf{s}})\Vert}{(h_0\gamma)^n}
  \bigg[\frac{\Vert \prod_{k=0}^{n-1}h(f_k(\emph{\textbf{s}}))R(f_n(\emph{\textbf{s}}))
  -\prod_{k=0}^{n-1}h(f_k(\emph{\textbf{s}}))R(\textbf{0})-g_n(\emph{\textbf{s}})P_0\Vert}
  {\Vert g_n(\emph{\textbf{s}})\Vert}\\
& &+ \frac{\Vert\prod_{k=0}^{n-1}h(f_k(\emph{\textbf{s}}))\widetilde{R}(f_n(\emph{\textbf{s}}))
  -\prod_{k=0}^{n-1}h(f_k(\emph{\textbf{s}}))
  \widetilde{R}(\textbf{0})-g_n(\emph{\textbf{s}})P_0\Vert}{\Vert g_n(\emph{\textbf{s}})\Vert} \bigg].\\
&\triangleq & \frac{\Vert g_n(\emph{\textbf{s}})\Vert}{(h_0\gamma)^n}
(\Vert r_n(\emph{\textbf{s}})\Vert +\Vert \widetilde{r}_n(\emph{\textbf{s}})\Vert).
\end{eqnarray*}
Applying the mean value theorem, $\Vert r_n(\emph{\textbf{s}})\Vert$ and
 $\Vert \widetilde{r}_n(\emph{\textbf{s}})\Vert$
converge to $0$ for $\emph{\textbf{s}}\in[0,1]^2\setminus\{\textbf{1}\}$ as $n\rightarrow\infty.$
And since $\lim\limits_{n\rightarrow\infty}\frac{\Vert g_n(\emph{\textbf{s}})\Vert}{(h_0\gamma)^n}
 =\Vert R(\emph{\textbf{s}})\Vert<\infty$, we can see that $\Vert R(\emph{\textbf{s}})-\widetilde{R}(\emph{\textbf{s}})\Vert=0$ and thus $R(\emph{\textbf{s}})$ is the unique solution.
\par
It remains to show \eqref{2.4a}.
  Set $R(\emph{\textbf{s}})=(R_1(\emph{\textbf{s}}),R_2(\emph{\textbf{s}}))$ and $\widetilde{Q}(\emph{\textbf{s}})=(\widetilde{Q}_1(\emph{\textbf{s}}),\widetilde{Q}_2(\emph{\textbf{s}}))$. We can see that $R_i(\emph{\textbf{s}})=\widetilde{Q}_i(\emph{\textbf{s}})S(\emph{\textbf{s}})$ for fixed $i=1,2$. From the proof of \cite[Theorem 1]{AV95}, one has
 $$
\widetilde{Q}(\emph{\textbf{s}})=\emph{\textbf{s}}{P_0}^\top
+\sum_{k=0}^{\infty}\frac{\hat{g}(f_k(\emph{\textbf{s}}))}{\gamma^{k+1}}{P_0}^\top,
 $$
 where $\top$ represents the transpose, $\hat{g}(\emph{\textbf{s}})=f(\emph{\textbf{s}})-s{A}^\top$ (all components are nonnegative) and $\sum\limits_{k=0}^{\infty}\frac{\hat{g}(f_k(\emph{\textbf{s}}))}{\gamma^{k+1}}<\infty$. Thus, there is a nonnegative sequence $\{r_{i,\emph{\textbf{j}}}\}$ such that $R_i(\emph{\textbf{s}})=\sum\limits_{\boldsymbol{j}\in\mathbb{N}^2\setminus\{\boldsymbol{0}\}}
r_{i,\emph{\textbf{j}}}\emph{\textbf{s}}^{\emph{\textbf{j}}}$. According to the proof of $\lim\limits_{n\to\infty}\widetilde{Q}_n=\widetilde{Q}(\emph{\textbf{s}})$ and $\lim\limits_{n\to\infty}S_n=S(\emph{\textbf{s}})$, for fixed $0\leq s_2<1$, $\frac{g_n^{(i)}(s_1,s_2)}{(p_1h_0)^n}$ can be expanded as a power series with respect to $s_1$ and converges to
$R_i(s_1,s_2)$ uniformly on $[0,1)$. It follows that
$$
\lim_{n\to\infty}\frac{g_n^{(i)}(s_1,s_2)}{(p_1h_0)^n}
=\sum_{\emph{\textbf{j}}=(j_1,j_2)\in\mathbb{N}^2\setminus\{\boldsymbol{0}\}}
\lim_{n\to\infty}\frac{P_i(X_n=\emph{\textbf{j}})}{(h_0\gamma)^n}s_1^{j_1}s_2^{j_2}
=\sum_{\emph{\textbf{j}}=(j_1,j_2)\in\mathbb{N}^2\setminus\{\boldsymbol{0}\}}
r_{i,\emph{\textbf{j}}}s_1^{j_1}s_2^{j_2}
$$
Comparing coefficients on both sides implies \eqref{2.4a}.
 The proof is complete.
 \hfill $\Box$
\end{proof}
\par
The next proposition is needed in the proof of Theorem \ref{th3.1}. We recall that $\boldsymbol{v}$ and $\boldsymbol{u}$ are strictly positive left and right eigenvectors of the mean matrix $M=(m_{ij})$, corresponding to its maximal eigenvalue $\rho$.
\begin{proposition}\label{pro2.4}
Let $\boldsymbol{\lambda}$ be the expected immigration vector, i.e., $\boldsymbol{\lambda}=E(I_1)$, and $J=\{\textbf{j}=(j_1,j_2)|j_1>0, j_2>0\}$. Then
 \begin{equation}
  \label{2.17}
  \lim_{n\rightarrow\infty}\sup_{\textbf{j}\in J}
  \Big\Vert\frac{\textbf{j}M^n+\sum_{i=0}^{n-1}\boldsymbol{\lambda} M^i}{\boldsymbol{u}\cdot(\textbf{j}+\boldsymbol{\lambda}\rho^{-1}
  +\cdots+\boldsymbol{\lambda}\rho^{-n})\rho^{n}}-\boldsymbol{v}\Big\Vert=0.
 \end{equation}
 \end{proposition}
\begin{proof} Before giving the proof of (\ref{2.17}), we first show that
\begin{equation} \label{2.18}
\lim_{n\rightarrow\infty}\sup\limits_{\emph{\textbf{x}}\in F}\Big\Vert \emph{\textbf{x}}\Big(\sum_{i=0}^{n-1}M^i\Big)\rho^{-n}
-\Big(\sum_{i=1}^n\rho^{-i}\Big)\boldsymbol{v}\Big\Vert=0,
\end{equation}
 where $F=\{\emph{\textbf{x}}=(x_1,x_2)|x_i>0,\boldsymbol{u}\cdot \emph{\textbf{x}}=1\}.$
Notice that for $1\leq k<n$, the norm in (\ref{2.18}) can be divided into two parts,
$$
\Big\Vert \frac{\boldsymbol{x}}{\rho^{n}}\Big(\sum\limits_{i=0}^{k-1}M^i\Big)
-\Big(\sum\limits_{i=n-k+1}^{n}\frac{1}{\rho^{i}}\Big)\boldsymbol{v}\Big\Vert
\ \
\text{and}
\ \
\Big\Vert \frac{\boldsymbol{x}}{\rho^{n}}\Big(\sum\limits_{i=k}^{n-1}M^i\Big)
-\Big(\sum\limits_{i=1}^{n-k}\frac{1}{\rho^{i}}\Big)\boldsymbol{v}\Big\Vert.
$$
It is known from the Frobenius theorem \cite[Lemma V.6.1]{AKN72} that
 \begin{equation}
 \label{2.19}
 \lim_{n\rightarrow\infty}\sup_{\emph{\textbf{x}}\in F}\big\Vert \emph{\textbf{x}}M^n\rho^{-n}-\boldsymbol{v}\big\Vert=0.
 \end{equation}
 Then, for any given $\varepsilon>0,$ there is $k_1\in\mathbb{N}$ such that for all $k>k_1,$ $\sup\limits_{\emph{\textbf{x}}\in F}\big\Vert \emph{\textbf{x}}M^k\rho^{-k}-\boldsymbol{v}\big\Vert<\frac{\varepsilon}{2}.$
 Thus, for $n>k_1,$
 $$
 \sup_{\emph{\textbf{x}}\in F}\Big\Vert \emph{\textbf{x}}M^{n-1}\rho^{-n}
 -\frac{\boldsymbol{v}}{\rho}\Big\Vert<\frac{\varepsilon}{2\rho},\ \ \cdots,\ \ \sup_{\emph{\textbf{x}}\in F}\Big\Vert \emph{\textbf{x}}M^{k_1}\rho^{-n}
 -\frac{\boldsymbol{v}}{\rho^{n-k_1}}\Big\Vert<\frac{\varepsilon}{2\rho^{n-k_1}}.
 $$
Therefore
$$
\sup_{\emph{\textbf{x}}\in F}\Big\Vert \emph{\textbf{x}}\Big(\sum_{i=k_1}^{n-1}M^i\Big)\rho^{-n}
-\Big(\sum_{i=1}^{n-k_1}\rho^{-i}\Big)\boldsymbol{v}\Big\Vert<\Big(\frac{1}{\rho}+\cdots
+\frac{1}{\rho^{n-k_1}}\Big)\frac{\varepsilon}{2}.
$$
Obviously, for given $k_1,$
$$
\lim_{n\rightarrow\infty}\sup_{\emph{\textbf{x}}\in F}
\Big\Vert \emph{\textbf{x}}\Big(\sum_{i=0}^{k_1-1}M^i\Big)\rho^{-n}
-\Big(\sum_{i=n-k_1+1}^{n}\rho^{-i}\Big)\boldsymbol{v}\Big\Vert=0,
$$
thus verifying (\ref{2.18}). Now, we return to the proof of (\ref{2.17}). Notice that
 $$
 \frac{\sum_{i=0}^{n-1}\boldsymbol{\lambda}M^i}{\boldsymbol{u}\cdot(\emph{\textbf{j}}
 +\frac{\boldsymbol{\lambda}}{\rho}+\cdots+\frac{\boldsymbol{\lambda}}{\rho^n})\rho^n}
 =\frac{\boldsymbol{u}\cdot(\frac{\boldsymbol{\lambda}}{\rho}+\cdots+\frac{\boldsymbol{\lambda}}{\rho^n})}
 {\boldsymbol{u}\cdot(\emph{\textbf{j}}
 +\frac{\boldsymbol{\lambda}}{\rho}+\cdots+\frac{\boldsymbol{\lambda}}{\rho^n})}
 \times\frac{\sum_{i=0}^{n-1}\boldsymbol{\lambda} M^i}{\boldsymbol{u}\cdot(\frac{\boldsymbol{\lambda}}{\rho}
 +\cdots+\frac{\boldsymbol{\lambda}}{\rho^n})\rho^n}
 $$
  and thus by (\ref{2.18}),
  \begin{equation}
  \label{2.20}
  \lim_{n\rightarrow\infty}\sup_{\emph{\textbf{j}}\in J}\Big\Vert\frac{\sum_{i=0}^{n-1}\boldsymbol{\lambda} M^i}{{\boldsymbol{u}\cdot(\emph{\textbf{j}}+\frac{\boldsymbol{\lambda}}{\rho}+\cdots
  +\frac{\boldsymbol{\lambda}}{\rho^n})\rho^n}}
  -\frac{(\boldsymbol{u}\cdot\frac{\boldsymbol{\lambda}}{\rho-1})\boldsymbol{v}}{\boldsymbol{u}\cdot(\emph{\textbf{j}}
  +\frac{\boldsymbol{\lambda}}{\rho-1})}\Big\Vert=0.
  \end{equation}
However,
$$
\frac{\emph{\textbf{j}}M^n}{{\boldsymbol{u}\cdot(\emph{\textbf{j}}+\frac{\boldsymbol{\lambda}}{\rho}+\cdots
+\frac{\boldsymbol{\lambda}}{\rho^n})\rho^n}}=\frac{(\boldsymbol{u}\cdot\emph{\textbf{j}} )(\emph{\textbf{j}}M^n)}{{\boldsymbol{u}\cdot(\emph{\textbf{j}}+\frac{\boldsymbol{\lambda}}{\rho}
+\cdots+\frac{\boldsymbol{\lambda}}{\rho^n})(\boldsymbol{u}\cdot\emph{\textbf{j}})\rho^n}}
$$
and by (\ref{2.19})
\begin{equation}
\label{2.21}
\lim_{n\rightarrow\infty}\sup_{\emph{\textbf{j}}\in J}\Big\Vert\frac{\emph{\textbf{j}}M^n}
{{\boldsymbol{u}\cdot(\emph{\textbf{j}}+\frac{\boldsymbol{\lambda}}{\rho}+\cdots
+\frac{\boldsymbol{\lambda}}{\rho^n})\rho^n}}
-\frac{(\boldsymbol{u}\cdot\emph{\textbf{j}})\boldsymbol{v}}{\boldsymbol{u}\cdot(\emph{\textbf{j}}+\frac{\boldsymbol{\lambda}}{\rho-1})}\Big\Vert=0.
\end{equation}
 Hence, (\ref{2.17}) follows from (\ref{2.20}) and (\ref{2.21}). The proof is complete.
 \hfill $\Box$
\end{proof}
\noindent
\text{\em{Proof of Theorem~\ref{th3.1}}}\ \ \ By conditioning on $X_n$, for any nonzero vector $\boldsymbol{l}$ satisfying $l_1\neq l_2$ and $\varepsilon>0$,
\begin{eqnarray*}
& &(h_0\gamma)^{-n}P_i\Big(\Big|\frac{\boldsymbol{l}\cdot X_{n+1}}{\textbf{1}\cdot X_n}-\frac{\boldsymbol{l}\cdot(X_nM)}{\textbf{1}\cdot X_n}\Big|>\varepsilon\Big)\\
&=& \sum_{\boldsymbol{j}\in\mathbb{N}^2\setminus\{\boldsymbol{0}\}}P(|\boldsymbol{l}\cdot X_{n+1}-\boldsymbol{l}\cdot(\emph{\textbf{j}}M)|
>\varepsilon(\textbf{1}\cdot \emph{\textbf{j}})|X_n=\emph{\textbf{j}})\frac{
P_i(X_n=\emph{\textbf{j}})}{(h_0\gamma)^n}\\
&=& \sum_{\boldsymbol{j}\in\mathbb{N}^2\setminus\{\boldsymbol{0}\}}\phi(\emph{\textbf{j}},\varepsilon)
\frac{P_i(X_n=\emph{\textbf{j}})}{(h_0\gamma)^n},
 \end{eqnarray*}
 where the last equation holds since the process $\{X_n\}$ is temporally homogeneous.
We shall show that there is $0<\lambda_0<1$ such that
 \begin{equation}\label{3.3}
\phi(\emph{\textbf{j}},\varepsilon)= O(\lambda_0^{\textbf{1}\cdot \emph{\textbf{j}}}),
\end{equation}
which means that there is a constant $C_{11}$ such that $\phi(\emph{\textbf{j}},\varepsilon)\leq C_{11}\lambda_0^{\textbf{1}\cdot \emph{\textbf{j}}}$. This yields that
$$
\phi(\emph{\textbf{j}},\varepsilon)\frac{P_i(X_n=\emph{\textbf{j}})}{(h_0\gamma)^n}
\leq C_{11}\lambda_0^{\textbf{1}\cdot \emph{\textbf{j}}}\frac{P_i(X_n=\emph{\textbf{j}})}{(h_0\gamma)^n}.
$$
By \eqref{2.4a}, we obtain that
$$
\lim_{n\to\infty}\phi(\emph{\textbf{j}},\varepsilon)\frac{P_i(X_n=\emph{\textbf{j}})}{(h_0\gamma)^n}
=\phi(\emph{\textbf{j}},\varepsilon)r_{i,\emph{\textbf{j}}}
$$
and
$$
\lim_{n\to\infty}C_{11}\lambda_0^{\textbf{1}\cdot \emph{\textbf{j}}}\frac{P_i(X_n=\emph{\textbf{j}})}{(h_0\gamma)^n}
=C_{11}\lambda_0^{\textbf{1}\cdot \emph{\textbf{j}}}r_{i,\emph{\textbf{j}}}.
$$
Since $\sum\limits_{\boldsymbol{j}\in\mathbb{N}^2\setminus\{\boldsymbol{0}\}}
\lambda_0^{\textbf{1}\cdot \emph{\textbf{j}}}r_{i,\emph{\textbf{j}}}=R_i(\lambda_0, \lambda_0)<\infty$, then by a slight modification of the dominated convergence theorem (see \cite{R87}) we have
$$
\lim_{n\to\infty}\sum_{\boldsymbol{j}\in\mathbb{N}^2\setminus\{\boldsymbol{0}\}}\phi(\emph{\textbf{j}},\varepsilon)
\frac{P_i(X_n=\emph{\textbf{j}})}{(h_0\gamma)^n}
=\sum_{\boldsymbol{j}\in\mathbb{N}^2\setminus\{\boldsymbol{0}\}}\phi(\emph{\textbf{j}},\varepsilon)r_{i,\emph{\textbf{j}}}.
$$
So the proof of \eqref{3.1} will be complete if we show \eqref{3.3}.
  Notice that for given $X_n=\emph{\textbf{j}}=(j_1,j_2)\in\mathbb{N}^2\setminus\{\boldsymbol{0}\}$, by the branching property,
\begin{equation}\label{2.15}
X_{n+1}=\sum_{r=1}^2\sum_{m=1}^{j_r}Z_{1,m,r}+I_{n+1},
\end{equation}
where for fixed $r$, $\{Z_{1,m,r}\}_{m\geq 1}$ are $i.i.d.$ $\mathbb{N}^2$ valued random vectors
 distributed as the population at time $1$ initiated by a particle of type $r$ at time $0$. $Z_{1,m,1}$, $Z_{1,m,2}$ and $I_{n+1}$ are all independent.
Recalling that $\{I_n\}_{n\geq1}$ are $i.i.d.$ Thus we have
\begin{equation}\label{3.10d}
\boldsymbol{l}\cdot X_{n+1}-\boldsymbol{l}\cdot(\emph{\textbf{j}}M)=\sum_{r=1}^2\sum_{m=1}^{j_r}\boldsymbol{l}\cdot (Z_{1,m,r}-e_rM)+\boldsymbol{l}\cdot I_{n+1}.
\end{equation}
 On the one hand, for $\alpha>1, \boldsymbol{l}=(l_1,l_2)$ and $\emph{\textbf{j}}=(j_1,j_2)$, applying Markov's inequality,
\begin{equation}\label{3.10b}
\begin{split}
 &{} P(\boldsymbol{l}\cdot X_{n+1}-\boldsymbol{l}\cdot(\emph{\textbf{j}}M)>\varepsilon(\textbf{1}\cdot \emph{\textbf{j}})|X_n=\emph{\textbf{j}})\\
 \leq &{}P\bigg(\sum_{r=1}^2\sum_{m=1}^{j_r}\boldsymbol{l}\cdot \big(Z_{1,m,r}-e_rM\big)>\frac{\varepsilon(\textbf{1}\cdot \emph{\textbf{j}})}{2}\Big)
+P\Big(\boldsymbol{l}\cdot I_1>\frac{\varepsilon(\textbf{1}\cdot \emph{\textbf{j}})}{2}\Big)\\
\leq & {}E\left[\alpha^{\sum_{r=1}^2\sum_{m=1}^{j_r}\boldsymbol{l}\cdot \big(Z_{1,m,r}-e_rM\big)}\right]\alpha^{-\frac{\varepsilon(\textbf{1}\cdot
 \emph{\textbf{j}})}{2}}+E\left(e^{\theta_0(\textbf{1}\cdot I_1)}\right)e^{-\frac{\theta_0\varepsilon(\textbf{1}\cdot
\emph{\textbf{j}})}{2\Vert \boldsymbol{l}\Vert}}\\
 =& {}\big[\Gamma_{1,(1)}(\alpha)\big]^{j_1}
\big[\Gamma_{1,(2)}(\alpha)\big]^{j_2}
+E\left(e^{\theta_0(\textbf{1}\cdot I_1)}\right)e^{-\frac{\theta_0\varepsilon(\textbf{1}\cdot
\emph{\textbf{j}})}{2\Vert \boldsymbol{l}\Vert}},
\end{split}
\end{equation}
where
$$
\Gamma_{1,(1)}(\alpha)=f_1^{(1)}(\alpha^{l_1},\alpha^{l_2})
\alpha^{-[\boldsymbol{l}\cdot(e_1M)+\frac{\varepsilon}{2}]},
\ \ \
\Gamma_{1,(2)}(\alpha)=f_1^{(2)}(\alpha^{l_1},\alpha^{l_2})
\alpha^{-[\boldsymbol{l}\cdot(e_2M)+\frac{\varepsilon}{2}]}.
$$
 Then, $\Gamma_{1,(1)}(1)=1$ and
$\lim\limits_{\alpha\downarrow1}\Gamma_{1,(1)}'(\alpha)=l_1m_{11}+l_2m_{12}-\left(\boldsymbol{l}\cdot(e_1M)+\frac{\varepsilon}{2}\right)
=-\frac{\varepsilon}{2}<0$.
Thus, there is $\alpha_0>1$ such that $0<\Gamma_{1,(1)}(\alpha_0)<1.$ Similar calculations hold for $\Gamma_{1,(2)}(\alpha)$. Since $E(e^{\theta_0(\textbf{1}\cdot I_1)})<\infty$ for some $\theta_0>0$, there are constants $C_{12}$ and $0<\lambda_3<1$ such that
$E\left(e^{\theta_0(\textbf{1}\cdot I_1)}\right)e^{-\frac{\theta_0\varepsilon(\textbf{1}\cdot
\emph{\textbf{j}})}{2\Vert \boldsymbol{l}\Vert}}\leq C_{12}{\lambda_3}^{\textbf{1}\cdot
\emph{\textbf{j}}}$.
Thus, there is $0<\lambda_4<1$ such that
\begin{equation}\label{3.5}
P(\boldsymbol{l}\cdot X_{n+1}-\boldsymbol{l}\cdot(\emph{\textbf{j}}M)>\varepsilon(\textbf{1}\cdot \emph{\textbf{j}})|X_n=\emph{\textbf{j}})
=O(\lambda_4^{\textbf{1}\cdot \emph{\textbf{j}}}).
\end{equation}
On the other side, for $0<\beta<1,$ by Markov inequality one has
\begin{equation}\label{3.10c}
\begin{split}
 & {}P\left(\boldsymbol{l}\cdot X_{n+1}-\boldsymbol{l}\cdot(\emph{\textbf{j}}M)<-\varepsilon(\textbf{1}\cdot \emph{\textbf{j}})|X_n=\emph{\textbf{j}}\right)\\
\leq&{} P\bigg(\sum_{r=1}^2\sum_{m=1}^{j_r}\boldsymbol{l}\cdot \big(Z_{1,m,r}-e_rM\big)<-\frac{\varepsilon(\textbf{1}\cdot \emph{\textbf{j}})}{2}\bigg)
+P\Big(\boldsymbol{l}\cdot I_1<-\frac{\varepsilon(\textbf{1}\cdot \emph{\textbf{j}})}{2}\Big)\\
 \leq &{} E\left[\beta^{\sum_{r=1}^2\sum_{m=1}^{j_r}\boldsymbol{l}\cdot \big(Z_{1,m,r}-e_rM\big)}\right]\beta^{\frac{\varepsilon(\textbf{1}\cdot j)}{2}}
+E\left(e^{\theta_0(\textbf{1}\cdot I_1)}\right)e^{-\frac{\theta_0\varepsilon(\textbf{1}\cdot
\emph{\textbf{j}})}{2\Vert \boldsymbol{l}\Vert}}\\
=& {} \big[\Gamma_{2,(1)}(\beta)\big]^{j_1}
\big[\Gamma_{2,(1)}(\beta)\big]^{j_2}
+E\left(e^{\theta_0(\textbf{1}\cdot I_1)}\right)e^{-\frac{\theta_0\varepsilon(\textbf{1}\cdot
\emph{\textbf{j}})}{2\Vert \boldsymbol{l}\Vert}},
\end{split}
\end{equation}
where $$
\Gamma_{2,(1)}(\beta)=f_1^{(1)}(\beta^{l_1},\beta^{l_2})
\beta^{-[\boldsymbol{l}\cdot(e_1M)-\frac{\varepsilon}{2}]},
\ \ \
\Gamma_{2,(2)}(\beta)=f_1^{(2)}(\beta^{l_1},\beta^{l_2})
\beta^{-[\boldsymbol{l}\cdot(e_2M)-\frac{\varepsilon}{2}]}.
$$
Then, $\Gamma_{2,(1)}(1)=1$ and
$\lim\limits_{\beta\uparrow1}\Gamma_{2,(1)}'(\beta)=\frac{\varepsilon}{2}>0.$
Then, there is $0<\beta_0<1$ such that $0<\Gamma_{2,(2)}(\beta_0)<1.$ Similar calculations hold for $\Gamma_{2,(2)}(\beta)$.
Thus, there is $0<\lambda_5<1$ such that
 \begin{equation}\label{3.6}
 P(\boldsymbol{l}\cdot X_{n+1}-\boldsymbol{l}\cdot(\emph{\textbf{j}}M)<-\varepsilon(\textbf{1}\cdot \emph{\textbf{j}})|X_n=\emph{\textbf{j}})
 =O(\lambda_5^{\textbf{1}\cdot \emph{\textbf{j}}}).
 \end{equation}
  Hence, combining (\ref{3.5}) and (\ref{3.6}) yields (\ref{3.3}).
\par
Now we turn to prove (\ref{3.2}). Let $k_0$ be fixed (to be chosen later),
by conditioning on $X_{n-k_0}$, for any nonzero vector $\boldsymbol{l}$ satisfying $l_1\neq l_2$ and $\varepsilon>0$,
\begin{eqnarray*}
& &(h_0\gamma)^{-n}P_i\Big(\Big|\frac{\boldsymbol{l}\cdot X_n}{\textbf{1}\cdot X_n}-\frac{\boldsymbol{l}\cdot \boldsymbol{v} }{\textbf{1}\cdot \boldsymbol{v} }\Big|>\varepsilon\Big)\\
&=& (h_0\gamma)^{-k_0}\sum_{\boldsymbol{j}\in\mathbb{N}^2\setminus\{\boldsymbol{0}\}}P\Big(\Big|\frac{\boldsymbol{l}\cdot X_n}{\textbf{1}\cdot X_n}-\frac{\boldsymbol{l}\cdot \boldsymbol{v}}{\textbf{1}\cdot \boldsymbol{v} }\Big|>\varepsilon\Big|X_{n-k_0}=\emph{\textbf{j}}\Big)\frac{P_i(X_{n-k_0}=\emph{\textbf{j}})}{(h_0\gamma)^{n-k_0}}\\
&=& \sum_{\boldsymbol{j}\in\mathbb{N}^2\setminus\{\boldsymbol{0}\}}\frac{\hat{\phi}(\emph{\textbf{j}},k_0,\varepsilon)}{(h_0\gamma)^{k_0}}
\frac{P_i(X_{n-k_0}=\emph{\textbf{j}})}{(h_0\gamma)^{n-k_0}},
 \end{eqnarray*}
where the last equation is due to homogeneity. We will show that there is $0<\lambda_6<1$ such that
\begin{equation}\label{3.7}
\hat{\phi}(\emph{\textbf{j}},k_0,\varepsilon)=O(\lambda_6^{\textbf{1}\cdot \emph{\textbf{j}}}).
\end{equation}
Then, there is a constant $\hat{C}_{11}>0$ such that
$$
\frac{\hat{\phi}(\emph{\textbf{j}},k_0,\varepsilon)}{(h_0\gamma)^{k_0}}
\frac{P_i(X_{n-k_0}=\emph{\textbf{j}})}{(h_0\gamma)^{n-k_0}}
\leq \hat{C}_{11}\lambda_6^{\textbf{1}\cdot \emph{\textbf{j}}}\frac{P_i(X_{n-k_0}=\emph{\textbf{j}})}{(h_0\gamma)^{n-k_0}}.
$$
By \eqref{2.4a}, we get that
$$
\lim_{n\to\infty}\frac{\hat{\phi}(\emph{\textbf{j}},k_0,\varepsilon)}{(h_0\gamma)^{k_0}}
\frac{P_i(X_{n-k_0}=\emph{\textbf{j}})}{(h_0\gamma)^{n-k_0}}
=\frac{\hat{\phi}(\emph{\textbf{j}},k_0,\varepsilon)}{(h_0\gamma)^{k_0}}r_{i,\emph{\textbf{j}}}
$$
and
$$
\lim_{n\to\infty}\hat{C}_{11}\lambda_6^{\textbf{1}\cdot \emph{\textbf{j}}}\frac{P_i(X_{n-k_0}=\emph{\textbf{j}})}{(h_0\gamma)^{n-k_0}}
=\hat{C}_{11}\lambda_6^{\textbf{1}\cdot \emph{\textbf{j}}}r_{i,\emph{\textbf{j}}}.
$$
Again by $\sum\limits_{\boldsymbol{j}\in\mathbb{N}^2\setminus\{\boldsymbol{0}\}}
\lambda_0^{\textbf{1}\cdot \emph{\textbf{j}}}r_{i,\emph{\textbf{j}}}=R_i(\lambda_0, \lambda_0)<\infty$, it follows from a slight modification of the dominated convergence theorem (see \cite{R87}) that
$$
\lim_{n\to\infty}\sum_{\boldsymbol{j}\in\mathbb{N}^2\setminus\{\boldsymbol{0}\}}
\frac{\hat{\phi}(\emph{\textbf{j}},k_0,\varepsilon)}{(h_0\gamma)^{k_0}}
\frac{P_i(X_{n-k_0}=\emph{\textbf{j}})}{(h_0\gamma)^{n-k_0}}
=\frac{1}{(h_0\gamma)^{k_0}}\sum_{\boldsymbol{j}\in\mathbb{N}^2\setminus\{\boldsymbol{0}\}}
\hat{\phi}(\emph{\textbf{j}},k_0,\varepsilon)r_{i,\emph{\textbf{j}}}.
$$
Thus, to prove \eqref{3.2}, it suffices to show \eqref{3.7}.
Conditioned on $X_{n-k_0}=\emph{\textbf{j}}=(j_1,j_2)\in \mathbb{N}^2\setminus\{\boldsymbol{0}\}$,
\begin{equation}\label{3.8}
 X_n=\sum_{r=1}^2\sum_{m=1}^{j_r}Z_{k_0,m,r}+\sum_{i=n-k_0+1}^{n}U_{n}^{(i)},
 \end{equation}
where for fixed $r,$ $\{Z_{k_0,m,r}\}_{m\geq 1}$ are $i.i.d.$ $\mathbb{N}^2$-valued random vectors
distributed as population at time $k_0$ initiated by a particle of type $r$ at time $0$.
$Z_{k_0,m,1}$, $Z_{k_0,m,2}$ and $U_{n}^{(i)}$ are all independent. Thus, from (\ref{3.8}) and $E(Z_n|Z_0)=Z_0M^n$ (see Athreya and Ney \cite{AKN72}, P184),
$$
E(\boldsymbol{l}\cdot X_n|X_{n-k_0}=\emph{\textbf{j}})
 =\boldsymbol{l}\cdot\Big(\emph{\textbf{j}}M^{k_0}+\sum_{i=0}^{k_0-1}\boldsymbol{\lambda} M^i\Big).
 $$
Note that for any $\varepsilon>0$ and $X_{n-k_0}=\emph{\textbf{j}}$,
$$
\frac{\boldsymbol{l}\cdot X_n}{\textbf{1}\cdot X_n}>\frac{\boldsymbol{l}\cdot \boldsymbol{v}}{\textbf{1}\cdot \boldsymbol{v} }+\varepsilon
$$
holds if and only if

\begin{eqnarray}\label{3.8a}
\frac{\tilde{\boldsymbol{l}}\cdot\left(X_n-\emph{\textbf{j}}M^{k_0}
-\sum_{i=0}^{k_0-1}\boldsymbol{\lambda}M^i\right)}
{\boldsymbol{u}\cdot\left(\emph{\textbf{j}}+\frac{\boldsymbol{\lambda}}{\rho}
+\cdots+\frac{\boldsymbol{\lambda}}{\rho^{k_0}}\right)\rho^{k_0}}
>\frac{-\tilde{\boldsymbol{l}}\cdot\left(\emph{\textbf{j}}M^{k_0}
+\sum_{i=0}^{k_0-1}\boldsymbol{\lambda}M^i\right)}
{{\boldsymbol{u}\cdot\left(\emph{\textbf{j}}+\frac{\boldsymbol{\lambda}}{\rho}
+\cdots+\frac{\boldsymbol{\lambda}}{\rho^{k_0}}\right)\rho^{k_0}}}
\end{eqnarray}
holds, where $\tilde{\boldsymbol{l}}=\boldsymbol{l}-\left(\frac{\boldsymbol{l}\cdot \boldsymbol{v}}{\textbf{1}\cdot \boldsymbol{v} }+\varepsilon\right)\textbf{1}.$ By Proposition \ref{pro2.4}, for each $\epsilon>0$ there is $k_0<\infty$ such that
\begin{equation}\label{3.9}
\sup_{\emph{\textbf{j}}\in J}\Big\Vert\frac{\emph{\textbf{j}}M^{k_0}+\sum_{i=0}^{{k_0}-1}\boldsymbol{\lambda} M^i}{{\boldsymbol{u}\cdot(\emph{\textbf{j}}+\frac{\boldsymbol{\lambda}}{\rho}+\cdots+\frac{\boldsymbol{\lambda}}
{\rho^{k_0}})\rho^{k_0}}}-\boldsymbol{v}\Big\Vert(2\Vert \boldsymbol{l} \Vert+\varepsilon)<\frac{\varepsilon(\textbf{1}\cdot \boldsymbol{v} )}{2}.
\end{equation}
By the Cauchy-Buniakowsky-Schwarz inequality, for any $\emph{\textbf{a}}=(a_1,a_2)$ and $\emph{\textbf{b}}=(b_1,b_2),$
\begin{equation}\label{3.10}
|\emph{\textbf{a}}\cdot \emph{\textbf{b}}|\leq(a_1^2+a_2^2)^{1/2}(b_1^2+b_2^2)^{1/2}
\leq2\Vert \emph{\textbf{a}}\Vert \Vert \emph{\textbf{b}}\Vert.
\end{equation}
From (\ref{3.9}) and (\ref{3.10}), we get that
$$\frac{-\tilde{\boldsymbol{l}}\cdot\left(\emph{\textbf{j}}M^{k_0}
+\sum_{i=0}^{k_0-1}\boldsymbol{\lambda}M^i\right)}{{\boldsymbol{u}\cdot(\emph{\textbf{j}}
+\frac{\boldsymbol{\lambda}}{\rho}
+\cdots+\frac{\boldsymbol{\lambda}}{\rho^{k_0}})\rho^{k_0}}}>-\tilde{\boldsymbol{l}}\cdot\boldsymbol{v}-\frac{\varepsilon(\textbf{1}\cdot \boldsymbol{v} )}{2}
=\frac{\varepsilon(\textbf{1}\cdot \boldsymbol{v} )}{2}.$$
This together with (\ref{3.8}) and (\ref{3.8a}) imply that for $\alpha>1$, $\emph{\textbf{j}}\neq \textbf{0},$
$\tilde{\boldsymbol{l}}=(\tilde{l}_1, \tilde{l}_2)$ and $\boldsymbol{u}=(u_1, u_2),$

\begin{equation}\label{3.10a}
\begin{split}
 &{} P\Big(\frac{\boldsymbol{l}\cdot X_n}{\textbf{1}\cdot X_n}-\frac{\boldsymbol{l}\cdot \boldsymbol{v}}{\textbf{1}\cdot \boldsymbol{v} }>\varepsilon|
X_{n-{k_0}}=\emph{\textbf{j}}\Big)\\
\leq &{} P\bigg(\frac{\tilde{\boldsymbol{l}}\cdot\big(X_n-\emph{\textbf{j}}M^{k_0}-\sum_{i=0}^{k_0-1}\boldsymbol{\lambda} M^i\big)}{\boldsymbol{u}\cdot\big(\emph{\textbf{j}}+\frac{\boldsymbol{\lambda}}{\rho}+\cdots+\frac{\boldsymbol{\lambda}}
{\rho^{k_0}}\big)\rho^{k_0}} >\frac{\varepsilon(\textbf{1}\cdot \boldsymbol{v} )}{2}\Big|X_{n-{k_0}}=\emph{\textbf{j}}\bigg)\\
=&{} P\bigg(\frac{\tilde{\boldsymbol{l}}\cdot\big[\sum_{r=1}^2\sum_{m=1}^{j_r}\big(Z_{k_0,m,r}-e_rM^{k_0}\big)
+\sum_{i=1}^{k_0}\big(U_{k_0}^{(i)}-\boldsymbol{\lambda}M^{k_0-i}\big)\big]}
{\boldsymbol{u}\cdot\big(\emph{\textbf{j}}
+\frac{\boldsymbol{\lambda}}{\rho}+\cdots+\frac{\boldsymbol{\lambda}}{\rho^{k_0}}\big)\rho^{k_0}}
>\frac{\varepsilon(\textbf{1}\cdot \boldsymbol{v} )}{2}\bigg)\\
 \leq &{} E\bigg[\alpha^{\sum_{r=1}^2\sum_{m=1}^{j_r}\tilde{\boldsymbol{l}}\cdot \big(Z_{k_0,m,r}-e_rM^{k_0}\big)
 +\sum_{i=1}^{k_0}\tilde{\boldsymbol{l}}\cdot \big(U_{k_0}^{(i)}-\boldsymbol{\lambda} M^{k_0-i}\big)}\bigg]
 \cdot \alpha^{-2^{-1}\varepsilon(\textbf{1}\cdot \boldsymbol{v} )\boldsymbol{u}\cdot\big(\emph{\textbf{j}}+\frac{\boldsymbol{\lambda}}
 {\rho}+\cdots+\frac{\boldsymbol{\lambda}}{\rho^{k_0}}\big)\rho^{k_0}}\\
=& {}\big[\Gamma_{3,(1)}(\alpha)\big]^{j_1}\big[\Gamma_{3,(1)}(\alpha)\big]^{j_2}\Gamma_4(\alpha),
\end{split}
\end{equation}
where $$
\Gamma_{3,(1)}(\alpha)=f_{k_0}^{(1)}(\alpha^{\tilde{l}_1},\alpha^{\tilde{l}_2})
\alpha^{-\big(\tilde{\boldsymbol{l}}\cdot (e_1M^{k_0})
+2^{-1}\varepsilon(\textbf{1}\cdot \boldsymbol{v} )u_1\rho^{k_0}\big)},
$$
$$
\Gamma_{3,(2)}(\alpha)=f_{k_0}^{(2)}(\alpha^{\tilde{l}_1},\alpha^{\tilde{l}_2})
\alpha^{-\big(\tilde{\boldsymbol{l}}\cdot (e_2M^{k_0})
+2^{-1}\varepsilon(\textbf{1}\cdot \boldsymbol{v} )u_2\rho^{k_0}\big)}
$$
and
$$
\Gamma_4(\alpha)=\Big[\prod_{i=1}^{k_0}h\big(f_{k_0-i}(\alpha^{\tilde{l}_1},\alpha^{\tilde{l}_2})\big)\Big]
\alpha^{-\big[\tilde{\boldsymbol{l}}\cdot \sum_{i=1}^{k_0}\boldsymbol{\lambda} M^{k_0-i}+2^{-1}\varepsilon(\textbf{1}\cdot \boldsymbol{v} )\boldsymbol{u}\cdot\big(\frac{\boldsymbol{\lambda}}
{\rho}+\cdots+\frac{\boldsymbol{\lambda}}{\rho^{k_0}}\big)\rho^{k_0}\big]}.
$$
Then, $\Gamma_{3,(1)}(1)=1$, $\Gamma_4(1)=1$, $\lim\limits_{\alpha\downarrow1}\Gamma_{3,(1)}'(\alpha)=-
2^{-1}\varepsilon(\textbf{1}\cdot \boldsymbol{v} )u_1\rho^{k_0}<0,$
 and
 $$
 \lim_{\alpha\downarrow1}\Gamma_4'(\alpha)=-2^{-1}\varepsilon(\textbf{1}\cdot \boldsymbol{v} )\boldsymbol{u}\cdot\big(\frac{\boldsymbol{\lambda}}
 {\rho}+\cdots+\frac{\boldsymbol{\lambda}}{\rho^{k_0}}\big)\rho^{k_0}<0.
 $$
Thus, there are $\alpha_1>1$ and $\alpha_2>1$ such that $0<\Gamma_{3,(1)}(\alpha_1)<1$ and $0<\Gamma_4(\alpha_2)<1$.
Similar calculations hold for $\Gamma_{3,(2)}(\alpha).$ Hence, there is a $0<\lambda_{7}<1$ such that
  \begin{equation}\label{3.11}
  P\left(\frac{\boldsymbol{l}\cdot X_n}{\textbf{1}\cdot X_n}>\frac{\boldsymbol{l}\cdot \boldsymbol{v}}{\textbf{1}\cdot \boldsymbol{v} }+\varepsilon\Big|X_{n-k_0}
  =\emph{\textbf{j}}\right)= O\left(\lambda_7^{\textbf{1}\cdot \emph{\textbf{j}}}\right).
  \end{equation}
\par
Similar techniques can be applied to the other side. Denote $\boldsymbol{l}^{\ast}=\boldsymbol{l}-(\frac{\boldsymbol{l}\cdot \boldsymbol{v} }{\textbf{1}\cdot \boldsymbol{v} }-\varepsilon)\textbf{1},$
 then for $\boldsymbol{l}^{\ast}=(l_1^{\ast}, l_2^{\ast})$
 and $0<\beta<1,$ by Markov's inequality,
\begin{eqnarray*}
& & P\Big(\frac{\boldsymbol{l}\cdot X_n}{\textbf{1}\cdot X_n}<\frac{\boldsymbol{l}\cdot \boldsymbol{v}}{\textbf{1}\cdot \boldsymbol{v} }-\varepsilon\Big|
    X_{n-{k_0}}=\emph{\textbf{j}}\Big)\\
&\leq & P\Big(\frac{\boldsymbol{l}^{\ast}\cdot (X_n-\emph{\textbf{j}}M^{k_0}-\sum_{i=0}^{k_0-1}\boldsymbol{\lambda} M^i)}{\boldsymbol{u}\cdot(\emph{\textbf{j}}+\frac{\boldsymbol{\lambda}}{\rho}+\cdots+\frac{\boldsymbol{\lambda}}{\rho^{k_0}})\rho^{k_0}}
    <-\frac{\varepsilon(\textbf{1}\cdot \boldsymbol{v} )}{2}\Big|X_{n-{k_0}}=\emph{\textbf{j}}\Big) \\
&=&P\bigg(\frac{\boldsymbol{l}^{\ast}\cdot \big[\sum_{r=1}^2\sum_{m=1}^{j_r}(Z_{k_0,m,r}-e_rM^{k_0})+\sum_{i=1}^{k_0}
   (U_{k_0}^{(i)}-\boldsymbol{\lambda}M^{k_0-i})\big]}{\boldsymbol{u}\cdot(\emph{\textbf{j}}+\frac{\boldsymbol{\lambda}}{\rho}+
   \cdots+\frac{\boldsymbol{\lambda}}{\rho^{k_0}})\rho^{k_0}}
   <-\frac{\varepsilon(\textbf{1}\cdot \boldsymbol{v} )}{2}\bigg)\\
& \leq & E\Big[\beta^{\sum_{r=1}^2\sum_{m=1}^{j_r}\boldsymbol{l}^{\ast}\cdot (Z_{k_0,m,r}-e_rM^{k_0})
    +\sum_{i=1}^{k_0}\boldsymbol{l}^{\ast}\cdot (U_{k_0}^{(i)}-\boldsymbol{\lambda} M^{k_0-i})}\Big]
    \beta^{2^{-1}\varepsilon(\textbf{1}\cdot \boldsymbol{v} )\boldsymbol{u}\cdot(\emph{\textbf{j}}
    +\frac{\boldsymbol{\lambda}}{\rho}+\cdots+\frac{\boldsymbol{\lambda}}{\rho^{k_0}})\rho^{k_0}}\\
&= & \big[\Gamma_{5,(1)}(\beta)\big]^{j_1}\big[\Gamma_{5,(2)}(\beta)\big]^{j_2}\Gamma_6(\beta),
 \end{eqnarray*}
where
$$
\Gamma_{5,(1)}(\beta)=f_{k_0}^{(1)}(\beta^{l^{\ast}_1},\beta^{l^{\ast}_2})
\beta^{-\left(\boldsymbol{l}^{\ast}\cdot(e_1M^{k_0})-2^{-1}\varepsilon(\textbf{1}\cdot \boldsymbol{v} )u_1\rho^{k_0}\right)},$$
 $$
 \Gamma_{5,(2)}(\beta)=f_{k_0}^{(2)}(\beta^{l^{\ast}_1},\beta^{l^{\ast}_2})\beta^{-\left(\boldsymbol{l}^{\ast}\cdot (e_2M^{k_0})
-2^{-1}\varepsilon(\textbf{1}\cdot \boldsymbol{v} )u_2\rho^{k_0}\right)}
 $$
 and
$$
 \Gamma_6(\beta)=\Big[\prod_{i=1}^{k_0}h(f_{k_0-i}(\beta^{l^{\ast}_1},\beta^{l^{\ast}_2}))\Big]
 \beta^{-\big(\boldsymbol{l}^{\ast}\cdot \sum_{i=1}^{k_0}\boldsymbol{\lambda} M^{k_0-i}-2^{-1}\varepsilon(\textbf{1}\cdot \boldsymbol{v} )\boldsymbol{u}\cdot(\frac{\boldsymbol{\lambda}}{\rho}
 +\cdots+\frac{\boldsymbol{\lambda}}{\rho^{k_0}})\rho^{k_0}\big)}.
 $$
Then, $\Gamma_{5,(1)}(1)=1$, $\Gamma_6(1)=1$,
 $\lim\limits_{\beta\uparrow1}\Gamma_{5,(1)}'(\beta)=2^{-1}\varepsilon(\textbf{1}\cdot \boldsymbol{v} )u_1\rho^{k_0}>0$
and
$$
   \lim_{\beta\uparrow1}\Gamma_6'(\beta)=2^{-1}\varepsilon(\textbf{1}\cdot \boldsymbol{v} )\boldsymbol{u}\cdot\Big(\frac{\boldsymbol{\lambda}}{\rho}
   +\cdots+\frac{\boldsymbol{\lambda}}{\rho^{k_0}}\Big)\rho^{k_0}>0.
     $$
 Thus, there are constants $0<\beta_1<1$ and $0<\beta_2<1$ such that $0<\Gamma_{5,(1)}(\beta_1)<1$ and $0<\Gamma_6(\beta_2)<1$. Similar calculations hold for $\Gamma_{5,(2)}(\beta).$
   Hence, there is $0<\lambda_{8}<1$ such that
   $$
  P\Big(\frac{\boldsymbol{l}\cdot X_n}{\textbf{1}\cdot X_n}
  <\frac{\boldsymbol{l}\cdot \boldsymbol{v}}{\textbf{1}\cdot \boldsymbol{v} }-\varepsilon\Big|X_{n-k_0}=\emph{\textbf{j}}\Big)
  = O(\lambda_{8}^{\textbf{1}\cdot \emph{\textbf{j}}}).
  $$
  This together with (\ref{3.11}) imply (\ref{3.7}).
 The proof is complete .
 \hfill $\Box$
 \par
We shall prove Theorem~\ref{th3.2} after stating the following lemma, which is due to Athreya and Vidyashankar \cite[Lemma 3]{AV95}.
\par
\noindent
\begin{lemma} \label{le3.1} Let $\{X_i\}_{i\geq1}$ be $i.i.d.$ with $E(X_1)=0, E(X_1^{2r})<\infty$ for some $r\geq1$ and $\bar{X_n}=\frac{\sum_{i=1}^nX_i}{n}$.
 Then, for all $n$,
  $$
 P(|\bar{X_n}|>\varepsilon)=O\left(n^{-r}\right),
 $$
 where $O$ means that there is a constant $C_{13}$ such that $P(|\bar{X_n}|>\varepsilon)\leq C_{13}n^{-r}$.
 \end{lemma}
\noindent
\text{\em{Proof of Theorem~\ref{th3.2}}}\ \ \ From (\ref{3.10b}) and (\ref{3.10c}), we see $\phi(\emph{\textbf{j}},\varepsilon)\leq \nu_1+\nu_2,$ where
$$
\nu_1=P\bigg(\frac{\big|\sum_{r=1}^2\sum_{m=1}^{j_r}\boldsymbol{l}\cdot (Z_{1,m,r}-e_rM)\big|}{\textbf{1}\cdot \emph{\textbf{j}}}
>\frac{\varepsilon}{2}\Big|X_n=\emph{\textbf{j}}\bigg)$$
and
$$\nu_2=P\Big(\frac{|\boldsymbol{l}\cdot I_1|}{\textbf{1}\cdot \emph{\textbf{j}}}>\frac{\varepsilon}{2}\Big|X_n=\emph{\textbf{j}}\Big).
$$
Notice that $\sum\limits_{r=1}^2\sum\limits_{m=1}^{j_r}\boldsymbol{l}\cdot \big(Z_{1,m,r}-e_rM\big)$ is a sum of two random walks with mean 0,
 then by Lemma \ref{le3.1}, $\nu_1=O\big( (\textbf{1}\cdot \emph{\textbf{j}})^{-s}\big).$
However, by Markov's inequality, there is a constant $C_{13}(\kappa,\boldsymbol{l},\varepsilon)$ such that
 $$
 \nu_2 \leq C_{13}(\kappa,\boldsymbol{l},\varepsilon)(\textbf{1}\cdot \emph{\textbf{j}})^{-\kappa},
 $$
 This, together with $\nu_1=O\big( (\textbf{1}\cdot \emph{\textbf{j}})^{-s}\big)$, yields that there is a constant $C_{14}(s,\kappa,\boldsymbol{l},\varepsilon)$ such that
$$
\phi(\emph{\textbf{j}},\varepsilon)\frac{P_i(X_n=\emph{\textbf{j}})}{(h_0\gamma)^n}
\leq C_{14}(s,\kappa,\boldsymbol{l},\varepsilon)(\textbf{1}\cdot \emph{\textbf{j}})^{-d}\frac{P_i(X_n=\emph{\textbf{j}})}{(h_0\gamma)^n}.
$$
where $d=\min\{s,\kappa\}$. Summing over $\emph{\textbf{j}}$ simultaneously on both sides yields that
$$
(h_0\gamma)^{-n}P_i\left(\left|\frac{\boldsymbol{l}\cdot X_{n+1}}{\textbf{1}\cdot X_n}-\frac{l\cdot(X_nM)}{\textbf{1}\cdot X_n}\right|>\varepsilon\right)\leq C_{14}(s,\kappa,\boldsymbol{l},\varepsilon)
(h_0\gamma)^{-n}E_i\left((\textbf{1}\cdot X_n)^{-{d}}\right).
$$
By a slight modification of the dominated convergence theorem, to show (\ref{3.1}), it remains to prove that
$$
\lim_{n\rightarrow\infty}(h_0\gamma)^{-n}E_i\left((\textbf{1}\cdot X_n)^{-{d}}\right)
$$
 exists and is finite. Set $\Gamma(d)=\int_0^{\infty}e^{-t}t^{d-1}dt,$ $1\leq d<\infty$. Using the change of variables, for any positive r.v. $X$ we get $\Gamma(d)E(X^{-d})=\int_0^{\infty}E(e^{-tX})t^{d-1}dt.$
Put $X=\textbf{1}\cdot X_n$, then
$$
\Gamma(d)E_i((\textbf{1}\cdot X_n)^{-d})=\int_0^{\infty}g_n^{(i)}(e^{-t},e^{-t})t^{d-1}dt, \ \  i=1,2.
$$
By (\ref{2.1}), for each $0<t<\infty,$
$$
(h_0\gamma)^{-n}g_n^{(i)}(e^{-t},e^{-t})\rightarrow R_i(e^{-t},e^{-t}) \ \ \ as \ n\rightarrow\infty.
$$
Then, there is $0<C_{15}<\infty$ such that for $n$ large enough and $0<t<\infty$,
$$
(h_0\gamma)^{-n}g_n^{(i)}(e^{-t},e^{-t})\leq C_{15}R_i(e^{-t},e^{-t}).
$$
If we can show that $\int_0^{\infty}R_i(e^{-t},e^{-t})t^{d-1}dt<\infty,$
then by the dominated convergence theorem
$$
\lim_{n\rightarrow\infty}\int_0^{\infty}(h_0\gamma)^{-n}g_n^{(i)}(e^{-t},e^{-t})t^{d-1}dt
=\int_0^{\infty}R_i(e^{-t},e^{-t})t^{d-1}dt<\infty.
$$
Indeed, $\int_1^{\infty}R_i(e^{-t},e^{-t})t^{d-1}dt<\infty$
since $\Vert R(\emph{\textbf{s}})\Vert=O(\Vert \emph{\textbf{s}}\Vert)$
for $\Vert \emph{\textbf{s}}\Vert<1$.
Hence it remains to prove that $\int_0^1R_i(e^{-t},e^{-t})t^{d-1}dt<\infty.$
 Define
$L_n=\int_{\rho^{-n}}^{\rho^{-n+1}}R_i(e^{-t},e^{-t})t^{d-1}dt.$
Then, by (\ref{2.2}), for $i,j=1,2,$
\begin{displaymath}
 \begin{split}
L_n={} &\rho^{-n}\int_1^{\rho}R_i(e^{-x\rho^{-n}},e^{-x\rho^{-n}})(x\rho^{-n})^{d-1}dx\\
={} &\prod_{k=0}^{n-1}h(f_k(\emph{\textbf{s}})){(h_0\gamma\rho^d)^{-n}}\int_1^{\rho}
    R_i\big(f_n^{(j)}(e^{-x\rho^{-n}}\textbf{1})\big)x^{d-1}dx\\
\leq {} & {(h_0\gamma\rho^d)^{-n}}\int_1^{\rho}R_i\big(f_n^{(j)}(e^{-x\rho^{-n}}\textbf{1})\big)x^{d-1}dx.
 \end{split}
\end{displaymath}
 Notice that
$f_n^{(j)}(e^{-x\rho^{-n}}\textbf{1})=E_j\big(e^{-x(\textbf{1}\cdot Z_n)\rho^{-n}}\big)$, which converges to $E_j\big(e^{-x(\textbf{1}\cdot \boldsymbol{v} )W}\big)$ uniformly for $x$ in $[1,\rho]$ as $n\rightarrow\infty$ since $\frac{Z_n}{\rho^n}$ converges to $\boldsymbol{v}W$ a.s. as $n\rightarrow\infty$ (see \cite[Theorem V.6.1]{AKN72}). Under the hypothesis that $E_i(\textbf{1}\cdot Z_1)^{2s}<\infty$ for some $s\geq1$ and $i=1,2$, we obtain that $W$ is nontrivial by \cite[Theorem V.6.1]{AKN72}. Thus,
 $\sup\limits_{n\geq1}R_i(f_n^{(j)}(e^{-x\rho^{-n}}\textbf{1}))<\infty$
  for any $1\leq x\leq \rho$ and $j=1,2.$ Hence by $h_0\gamma\rho^d>1$,
  $$
  \int_0^1R_i(e^{-t},e^{-t})t^{d-1}dt=\sum_{n\geq1}L_n<\infty.
  $$
The proof of (\ref{3.1}) is complete. Now we turn to the proof of (\ref{3.2}). For given $X_{n-k_0}=\emph{\textbf{j}}$, from (\ref{3.10a}),
$$P\Big(\frac{\boldsymbol{l}\cdot X_n}{\textbf{1}\cdot X_n}-\frac{\boldsymbol{l}\cdot \boldsymbol{v}}{\textbf{1}\cdot \boldsymbol{v} }
    >\varepsilon\big|X_{n-{k_0}}=\emph{\textbf{j}}\Big)\leq \nu_3+\nu_4,
    $$
where
$$
\nu_3=P\bigg(\frac{\tilde{\boldsymbol{l}}\cdot\big[\sum_{r=1}^2\sum_{m=1}^{j_r}(Z_{k_0,m,r}-e_rM^{k_0})\big]}
    {\boldsymbol{u}\cdot\big(\emph{\textbf{j}}+\boldsymbol{\lambda}\rho^{-1}
    +\cdots+\boldsymbol{\lambda}\rho^{-k_0}\big)\rho^{k_0}}
    >\frac{\varepsilon(\textbf{1}\cdot \boldsymbol{v} )}{4}\bigg)
$$
and
$$
\nu_4=P\bigg(\frac{\sum_{i=1}^{k_0}\tilde{\boldsymbol{l}}\cdot\big(U_{k_0}^{(i)}-\boldsymbol{\lambda} M^{k_0-i}\big)}{\boldsymbol{u}\cdot\big(\emph{\textbf{j}}+\boldsymbol{\lambda}\rho^{-1}
    +\cdots+\boldsymbol{\lambda}\rho^{-k_0}\big)\rho^{k_0}}
  >\frac{\varepsilon(\textbf{1}\cdot \boldsymbol{v} )}{4}\bigg).
$$
Notice that $\sum_{r=1}^2\sum_{m=1}^{j_r}(Z_{k_0,m,r}-e_rM^{k_0})$ is also a sum of two random walks with mean 0,
then by Lemma \ref{le3.1},
$\nu_3=O\big( (\textbf{1}\cdot \emph{\textbf{j}})^{-s}\big)$.
By Markov's inequality,
 $\nu_4\leq C_{16}(s,\kappa,\boldsymbol{l},\epsilon)(\textbf{1}\cdot \emph{\textbf{j}})^{-\kappa}.$ It follows that
$$
\hat{\phi}(\emph{\textbf{j}},\varepsilon)\frac{P_i(X_{n-k_0}=\emph{\textbf{j}})}{(h_0\gamma)^n}
=O\big((\textbf{1}\cdot \emph{\textbf{j}})^{-d}\big)\frac{P_i(X_{n-k_0}=\emph{\textbf{j}})}{(h_0\gamma)^{n}},
$$
where $d=\min\{s,\kappa\}$. Summing over $\emph{\textbf{j}}$ simultaneously on both sides yields that
$$
(h_0\gamma)^{-n}P_i\Big(\frac{\boldsymbol{l}\cdot X_n}{\textbf{1}\cdot X_n}-\frac{\boldsymbol{l}\cdot \boldsymbol{v} }{\textbf{1}\cdot \boldsymbol{v} }>\varepsilon\Big)=(h_0\gamma)^{-n}O\big(E_i(\textbf{1}\cdot X_{n-{k_0}})^{-d}\big).
$$
 Similarly, we can show that
$\lim\limits_{n\rightarrow\infty}(h_0\gamma)^{-(n-k_0)}E_i\big((\textbf{1}\cdot X_{n-k_0})^{-d}\big)$
 exists and is finite. Therefore, (\ref{3.2}) follows from a slight modification of the dominated convergence theorem.
 The proof is complete.\hfill $\Box$
\par
Turning to the proof of Theorem~\ref{th3.3}, we again start with a proposition. The following proposition extends \cite[Theorem 3]{AV95} and is a technical result needed for Theorem~\ref{th3.3}.
\par
\noindent
\begin{proposition}\label{pro2.2} Assume that (\ref{2.7}) and (\ref{2.8}) hold, then for $\boldsymbol{s}\in(0,1]^2$
 \begin{equation}\label{2.9}
 \lim_{n\rightarrow\infty}\frac{\log g_n^{(i)}(\textbf{s})}{k_i^n}= G_i(\textbf{s})
  \end{equation}
exists and is the solution of the vector functional equations
\begin{equation}\label{2.10}
G_i(f(\textbf{s}))=k_iG_i(\textbf{s}),\ \
\lim_{\textbf{s}\downarrow\emph{\textbf{0}}}G_i(\textbf{s})=-\infty
\end{equation}
subject to
$$
 G_i(\emph{\textbf{1}})=0,\ \ \ \  -\infty <G_i(\textbf{s})< 0\ \ \
 for \ \boldsymbol{s}\in(0,1]^2\setminus\{\boldsymbol{1}\}.
$$
\end{proposition}
\begin{proof}\ According to (\ref{2.7}) and (\ref{2.8}), $k_1, k_2, d_1, d_2$ are well defined. They are finite and could not be equal to 0. Note that for $\boldsymbol{s}\in(0,1]^2$
\begin{displaymath}
 \begin{split}
f_1^{(1)}(s_1,s_2)={} & s_1^{k_1}P_1(k_1,0)+\sum_{j\geq1}s_1^{k_1}s_2^jP_1(k_1,j)
+\sum_{i>k_1}\sum_{j\geq0}s_1^is_2^jP_1(i,j)\\
={}& s_1^{k_1}p\left[1+\delta_1 a(s_2)+\delta_2 b(s_1,s_2)\right],
 \end{split}
\end{displaymath}
where
$$
p=P_1(k_1,0),\  \ \delta_1=\frac{\alpha_1}{p}, \ \ \delta_2=\frac{\alpha_2}{p},\ \
\alpha_1=\sum_{j\geq1}P_1(k_1,j),
$$
$$
\alpha_2=1-p-\alpha_1,\ \
 a(s_2)=\sum_{j\geq1}\frac{P_1(k_1,j)}{\alpha_1}s_2^j,\ \
 b(s_1,s_2)=\sum_{i>k_1}\sum_{j\geq0}\frac{P_1(i,j)}{\alpha_2}s_1^{i-k_1}s_2^j.
$$
Indeed, $a(s_2)$ and $b(s_1,s_2)$ are probability generating functions. Thus
$$
f_{n+1}^{(1)}(s_1,s_2)=\left(f_n^{(1)}(s_1,s_2)\right)^{k_1}p\left[1+\delta_1 a(f_n^{(2)}(s_1,s_2))
  +\delta_2 b(f_n(s_1,s_2))\right].
$$
 Let $L_n(s_1,s_2)=\big(f_n^{(1)}(s_1,s_2)\big)^{k_1^{-n}}$.
 Then
 \begin{equation}\label{2.11}
 L_{n+1}(s_1,s_2)=L_n(s_1,s_2)\big(F_n(s_1,s_2)\big)^{k_1^{-(n+1)}},
 \end{equation}
where
\begin{equation}\label{2.12}
F_n(s_1,s_2)=p\left[1+\delta_1 a(f_n^{(2)}(s_1,s_2))+\delta_2 b(f_n(s_1,s_2))\right].
\end{equation}
Iterating (\ref{2.11}) yields that
$$
L_n(s_1,s_2)=s_1\prod_{j=0}^{n-1}\big(F_j(s_1,s_2)\big)^{k_1^{-(j+1)}}.
$$
This implies that
 \begin{equation}\label{2.13}
 f_n^{(1)}(s_1,s_2)=s_1^{k_1^n}\Big(\prod_{j=0}^{n-1}\big(F_j(s_1,s_2)\big)^{k_1^{-(j+1)}}\Big)^{k_1^n}
\end{equation}
 and
\begin{equation}\label{2.14}
  \frac{\log f_n^{(1)}(\emph{\textbf{s}})}{k_1^n}=\log s_1+\sum_{j=0}^{n-1}k_1^{-(j+1)}\log(F_j(s_1,s_2)).
 \end{equation}
Notice that $\sum_{j=0}^{n}k_1^{-(j+1)}\log(F_j(s_1,s_2))$ is a Cauchy sequence since
\begin{displaymath}
 \begin{split}
\big|\log F_j(s_1,s_2)\big|\leq{} & |\log p|+\log\left[1+\delta_1 a(f_j^{(2)}(s_1,s_2))+\delta_2 b(f_j(s_1,s_2))\right]\\
\leq{} & |\log p|+\log(1+\delta_1+\delta_2)<\infty.
 \end{split}
\end{displaymath}
Therefore (\ref{2.14}) converges uniformly for $\emph{\textbf{s}}$ in $[0,1]^2$
 as $n\rightarrow\infty$, we denote the limit by $H_1(\emph{\textbf{s}})$. Note that for $\boldsymbol{s}\in(0,1]^2$
 \begin{displaymath}
 \begin{split}
h(s_1,s_2)={} &s_1^{d_1}P\big(I_1=(d_1,0)\big)+\sum_{j\geq1}s_1^{d_1}s_2^jP\big(I_1=(d_1,j)\big)
+\sum_{i>d_1}\sum_{j\geq0}s_1^is_2^jP(I_1=(i,j))\\
={} & s_1^{d_1}p^{\ast}\left(1+\sigma_1 c(s_2)+\sigma_2 d(s_1,s_2)\right),
 \end{split}
\end{displaymath}
where
$$
p^{\ast}=P(I_1=(d_1,0)),\ \sigma_1=\frac{\beta_1}{p^\ast},\  \sigma_2=\frac{\beta_2}{p^\ast},\
 \beta_1=\sum_{j\geq1}P\big(I_1=(d_1,j)\big),
$$
$$
\beta_2=1-p^{\ast}-\beta_1, \ c(s_2)=\sum_{j\geq1}\frac{P\big(I_1=(d_1,j)\big)}{\beta_1}s_2^j,\ \ d(s_1,s_2)=\sum_{i>d_1}\sum_{j\geq0}\frac{P\big(I_1=(i,j)\big)}{\beta_2}s_1^{i-d_1}s_2^j.
$$
Notice that $c(s_2)$ and $d(s_1,s_2)$ are generating functions. From (\ref{1.2}) and (\ref{2.13}),
we obtain
 \begin{displaymath}
 \begin{split}
& g_n^{(1)}(s_1,s_2)\\={} & f_n^{(1)}(s_1,s_2)\prod_{r=0}^{n-1}h\left(f_r^{(1)}(s_1,s_2),
f_r^{(2)}(s_1,s_2)\right)\\
={} & f_n^{(1)}(s_1,s_2)\prod_{r=0}^{n-1}(f_r^{(1)}(\emph{\textbf{s}}))^{d_1}p^{\ast}
\left(1+\sigma_1 c(f_r^{(2)}(\emph{\textbf{s}}))+\sigma_2 d\left(f_r(\emph{\textbf{s}})\right)\right)\\
= {} & f_n^{(1)}(s_1,s_2)\prod_{r=0}^{n-1}\Big[s_1^{k_1^r}
\Big(\prod_{j=0}^{r-1}(F_j(s_1,s_2))^{k_1^{-(j+1)}}\Big)^{k_1^r}\Big]^{d_1}
p^{\ast}\Big(1+\sigma_1 c(f_r^{(2)}(\emph{\textbf{s}}))+\sigma_2 d\big(f_r(\emph{\textbf{s}})\big)\Big).
 \end{split}
\end{displaymath}
 Therefore we get
 \begin{displaymath}
 \begin{split}
 \frac{\log g_n^{(1)}(\emph{\textbf{s}})}{k_1^n}= {} & \frac{\log f_n^{(1)}(\emph{\textbf{s}})}{k_1^n}
 +\frac{d_1\left(\sum_{r=0}^{n-1}k_1^r\left(\log s_1+\sum_{j=0}^{r-1}\frac{1}{k_1^{j+1}}
 \log F_j(s_1,s_2)\right)\right)}{k_1^n}\\
 {} &+\frac{\sum_{r=0}^{n-1}\log p^{\ast}\left(1+\sigma_1 c(f_r^{(2)}(\emph{\textbf{s}}))
 +\sigma_2 d(f_r(\emph{\textbf{s}}))\right)}{k_1^n}\\
 \triangleq {} & \eta_{n1}+\eta_{n2}+\eta_{n3}.
\end{split}
\end{displaymath}
 Note that $\eta_{n1}$ converges to $H_1(\emph{\textbf{s}})$
 and $\eta_{n3}$ converges to 0 as $n\rightarrow\infty$. However,
  \begin{displaymath}
 \begin{split}
 \eta_{n2}={} &\frac{d_1\sum_{r=0}^{n-1}k_1^r\log s_1}{k_1^n}+\frac{d_1\sum_{r=0}^{n-1}k_1^r
 \sum_{j=0}^{r-1}\frac{1}{k_1^{j+1}}\log F_j(s_1,s_2)}{k_1^n},\\
\end{split}
\end{displaymath}
 the second term of the right side converges since $\sum_{j=0}^nk_1^{-(j+1)}\log(F_j(s_1,s_2))$
 is a cauchy sequence and the first converges obviously. Therefore we obtain (\ref{2.9}). From (\ref{2.15}),
$g_{n+1}^{(1)}(\emph{\textbf{s}})=h(\emph{\textbf{s}})g_n^{(1)}(f(\emph{\textbf{s}}))$ and thus
$$
\frac{\log g_{n+1}^{(1)}(\emph{\textbf{s}})}{k_1^{n+1}}=\frac{\log h(\emph{\textbf{s}})
+\log g_n^{(1)}(f(\emph{\textbf{s}}))}{k_1^{n+1}}.
$$
 Letting $n\rightarrow\infty$ yields $G_1(f(\boldsymbol{s}))=k_1G_1(\emph{\textbf{s}})$. Obviously, $G_1(\boldsymbol{1})=0$. For $\boldsymbol{s}\in(0,1]^2\setminus\{\boldsymbol{1}\}$, $0< g_n^{(1)}(\boldsymbol{s})<1$ and thus $-\infty<G_i(\boldsymbol{s})<0$. We can find $\emph{\textbf{s}}_n\downarrow \boldsymbol{0}$ as $n\uparrow\infty$ such that $f(\emph{\textbf{s}}_n)\downarrow \boldsymbol{0}$. Note that $G_1(f(\emph{\textbf{s}}_n))=k_1G_1(\emph{\textbf{s}}_n)$. Letting $n\rightarrow\infty$ gives that
$\lim_{\emph{\textbf{s}}\downarrow\textbf{0}}G_1(\emph{\textbf{s}})=-\infty$.
 Similar arguments give the result for $i=2$, and are therefore omitted.
\hfill $\Box$
\end{proof}
\begin{remark}
In the special case $h_0=1$ and $k_i=2$ for $i=1,2$, $g_n^{(i)}(\boldsymbol{s})=f_n^{(i)}(\boldsymbol{s})$ and $G_i(\boldsymbol{s})=\lim\limits_{n\rightarrow\infty}\frac{\log f_n^{(i)}(\boldsymbol{s})}{2^n}$, which has been studied by Athreya and Vidyashankar (see \cite{AV95} or \cite{AV97}).
\end{remark}
\noindent
\text{\em{Proof of Theorem~\ref{th3.3}}}\ \ \ \ We only show the case $i=1$. Similar discussion can be conducted for the case $i=2$. It is known from (\ref{3.3}) and (\ref{3.7})
  that there are $0<s_{\varepsilon}<1$,
$0<s_{\varepsilon}^{\ast}<1$ and constants $C_1(\varepsilon)$, $C_2(\varepsilon)$ such that
$$
P_1\Big(\Big|\frac{\boldsymbol{l}\cdot X_{n+1}}{\textbf{1}\cdot X_n}-\frac{\boldsymbol{l}\cdot(X_nM)}{\textbf{1}\cdot X_n}\Big|>\varepsilon\Big)
\leq C_1(\varepsilon)\sum_{j\in \mathbb{N}^2 }s_{\varepsilon}^{\textbf{1}\cdot \emph{\textbf{j}}}P_1(X_n=\emph{\textbf{j}})\\
= C_1(\varepsilon)g_n^{(1)}(s_{\varepsilon}\textbf{1}),
$$
 and
 $$P_1\Big(\Big|\frac{\boldsymbol{l}\cdot X_n}{\textbf{1}\cdot X_n}-\frac{\boldsymbol{l}\cdot \boldsymbol{v} }{\textbf{1}\cdot \boldsymbol{v} }\Big|>\varepsilon\Big)\leq C_2(\varepsilon)g_{n-k_0}^{(1)}(s_{\varepsilon}^{\ast}\textbf{1}).$$
 Then (\ref{3.13}) and (\ref{3.14}) follow from Proposition~\ref{pro2.2}.
 The proof is complete. \hfill $\Box$

 To prove Theorem~\ref{th3.4}, we shall use the following two results.
\par
\noindent
\begin{proposition}\label{pro2.3} Let $\boldsymbol{u}$ be the right eigenvector of $M$ corresponding to
its maximal eigenvalue $\rho$ $(\rho>1)$ and $E(I_n)=\boldsymbol{\lambda}=(\lambda_1, \lambda_2).$
Assume that $\lambda_i<\infty$, $m_{ij}<\infty$ for $i,j=1,2.$
Define $\mathcal{F}_n=\sigma (X_0,\cdots, X_n)$ and
\begin{equation}\label{2.16}
Y_n=\rho^{-n}\Big[\boldsymbol{u}\cdot X_n-\frac{\rho^{n+1}-1}{\rho-1}(\boldsymbol{u}\cdot \boldsymbol{\lambda})\Big].
\end{equation}
 Then $\{(Y_n,\mathcal{F}_n);n\geq0\}$ is a martingale
 and converges to a $r.v.$ $Y$ w.p.1 as $n\rightarrow\infty.$
\end{proposition}
\begin{proof}\ From (\ref{2.15}),
 we obtain that
$ E_i(\boldsymbol{u} \cdot X_{n+1}|X_n)=\boldsymbol{u} \cdot (X_nM)+\boldsymbol{u} \cdot \boldsymbol{\lambda}$
 and thus
 \begin{displaymath}
 \begin{split}
 E_i[Y_{n+1}|\mathcal{F}_n]={} & E_i\Big[\rho^{-(n+1)}\Big( \boldsymbol{u} \cdot X_{n+1}-\frac{\rho^{n+2}-1}{\rho-1}
 (\boldsymbol{u}\cdot \boldsymbol{\lambda})\Big)\Big|X_n\Big]\\
               ={} & \rho^{-n-1}\Big[\boldsymbol{u} \cdot ( X_nM)+ \boldsymbol{u}\cdot \boldsymbol{\lambda}
                     -\frac{\rho^{n+2}-1}{\rho-1}( \boldsymbol{u}\cdot \boldsymbol{\lambda})\Big]\\
               ={} & \rho^{-n}\Big[\boldsymbol{u}\cdot X_n-\frac{\rho^{n+1}-1}{\rho-1}
                     (\boldsymbol{u}\cdot \boldsymbol{\lambda} )\Big]\\
               ={} & Y_n,\\
 \end{split}
\end{displaymath}
where we have used $\boldsymbol{u} \cdot ( X_nM)=X_n(M{\boldsymbol{u}}^\top)$ and ${\boldsymbol{u}}^\top$ is the transpose of $\boldsymbol{u}.$
Hence $\{(Y_n,\mathcal{F}_n);n\geq0\}$ is a martingale. Since $\lambda_i<\infty$, $m_{ij}<\infty$, then $E_i(Y_n)=E_i(Y_1)<\infty$. By the definition of $Y_n$, one has
$$
Y_n=\frac{\boldsymbol{u}\cdot X_n}{\rho^n}-(\boldsymbol{u}\cdot \boldsymbol{\lambda})\sum_{l=0}^n\rho^{-l},
$$
where $\sum\limits_{l=0}^{\infty}\rho^{-l}<\infty$ since $\rho>1$. This, together with $\sup\limits_{n}E_i(Y_n)=E_i(Y_1)<\infty$ implies that $\sup\limits_{n}E_i\left(\frac{\boldsymbol{u}\cdot X_n}{\rho^n}\right)<\infty$. Thus, $\sup\limits_{n}E_i[|Y_n|]<\infty$. By the martingale convergence theorem (see \cite[Theorem 7.11]{KFC}), $Y_n$ converges to a $r.v.$ $Y$ w.p.1 as $n\rightarrow\infty.$\hfill $\Box$
\end{proof}

\par
The following lemma is proved in the no-immigration case (see \cite[Lemma 2.2]{AV97}). We refer the reader to \cite{AV97} for a proof.
\par
\noindent
\begin{lemma} \label{le3.2}
 Suppose that $E_i\big(e^{\theta_0(\emph{\textbf{1}}\cdot Z_1)}\big)<\infty$ for some $\theta_0>0$ and $i=1,2$. Let $\boldsymbol{u}=(u_1, u_2)$ be the right eigenvector of eigenvalue $\rho$. For $\theta\leq\theta_0$, define
 $$
 \Phi_i(\theta)=E_i\left(e^{\theta\left(\boldsymbol{u}\cdot Z_1-u_r\rho\right)}\right) \ \ for \  i=1,2.
 $$
 Then, there exists a constant $0<C_{11}<\infty$ such that
$$
\Phi_i(\theta)\leq 1+C_{11}\theta^2 \ \ for  \  i=1,2.
$$
\end{lemma}
\noindent
 \text{\em{Proof of Theorem~\ref{th3.4}}}\ \ \ Set $B_n=\rho^{-n}\boldsymbol{u}\cdot X_n$. By the proof of Proposition \ref{pro2.3}, $\{Y_n\}$ is a martingale and $\sup\limits_{n}E_i\left(B_n\right)<\infty$. By $E_i(Y_{n+1}|\mathcal{F}_n)=Y_n$, one has
 $$
 E_i\left(B_{n+1}|\mathcal{F}_n\right)=B_n-\frac{\rho^{n+1}-1}{\rho^n(\rho-1)}\left(\boldsymbol{u}\cdot \boldsymbol{\lambda}\right)+\frac{\rho^{n+2}-1}{\rho^{n+1}(\rho-1)}\left(\boldsymbol{u}\cdot \boldsymbol{\lambda}\right)
              =B_n+\frac{\boldsymbol{u}\cdot \boldsymbol{\lambda}}{\rho^{n+1}}
             \geq B_n.
 $$
 This implies that $\{B_n\}$ is a submartingale. Then, for $\theta>0,$ $e^{\theta B_n}$ is also a submartingale. Note that $E_i\big(e^{\theta_0(\boldsymbol{1}\cdot X_1)}\big)<\infty$ for some $\theta_0>0$, so there is a $\theta_1>0$ such that $E_i\big(e^{\theta_1B_1}\big)=E_i\big(e^{\theta_1\rho^{-1}\boldsymbol{u}\cdot X_1}\big)<\infty$. Since the expectation of a submartingale is non-decreasing (see \cite{KFC}), we get $E_i\big(e^{\theta_1B_1}\big)\leq E_i\big(e^{\theta_1B_n}\big)$ for any $n\geq 2$. We notice that $E_i\big(e^{\theta B_n}\big)$ is monotone in $\theta$. Thus, for any $n\geq 2$, there is $0<\theta_n<\theta_1$ such that $ E_i\big(e^{\theta_1B_1}\big)=E_i\big(e^{\theta_nB_n}\big)<E_i\big(e^{\theta_1B_n}\big)$. Hence, we can find a sequence $\{\theta_n\}_{n\geq1}$ such that
$$
K\equiv E_i\big(e^{\theta_1B_1}\big)=E_i(e^{\theta_nB_n})<\infty.
$$
By the conditional Jensen inequality,
$$
K\equiv E_i\big(e^{\theta_nB_n}\big)=E_i\left(E_i(e^{\theta_{n+1}B_{n+1}}|\mathcal{F}_n)\right)\geq E_i(e^{\theta_{n+1}B_n}).
$$
It implies that $\{\theta_n\}_{n\geq1}$ is a decreasing sequence and thus
$\lim\limits_{n\rightarrow \infty} \theta_n=\theta_0^\star$ exists.
To complete the proof, it suffices to show that $\theta_0^\star>0.$ Noting that
\begin{equation}\label{3.16}
E_i\big(e^{\theta_{n+1}B_{n+1}}\big)=E_i\left( e^{\theta_{n+1}B_n}E_i(e^{\theta_{n+1}(B_{n+1}-B_n)}|\mathcal{F}_n)\right)
\end{equation}
and
$$
B_{n+1}-B_n=\rho^{-(n+1)}\Big[\sum_{r=1}^{2}\sum_{j=1}^{X_n^{(r)}}\left(\boldsymbol{u}\cdot Z_{1,j,r}-u_r\rho\right)+\boldsymbol{u}\cdot I_{n+1}\Big],
$$
where for fixed $r,$ $\{Z_{1,j,r}\}_{j\geq1}$ are $i.i.d.$ $\mathbb{N}^2$ valued random vectors distributed as population at time 1 initiated by a particle of type $r$ at time 0. $Z_{1,1,1}$ and $Z_{1,1,2}$ are independent. Put $\Phi_r(\theta)=E_r(e^{\theta(\boldsymbol{u}\cdot Z_1-u_r\rho)})$ for $r=1,2$. Therefore
 \begin{displaymath}
 \begin{split}
 E_i\left(e^{\theta_{n+1}(B_{n+1}-B_n)}|\mathcal{F}_n\right)
               = {} &\prod_{r=1}^2\Big[\Phi_r(\frac{\theta_{n+1}}{\rho^{n+1}})\Big]^{X_n^{(r)}}
                E\Big(e^{\frac{\theta_{n+1}}{\rho^{n+1}}(\boldsymbol{u}\cdot I_1 )}\Big)\\
              \leq {} & e^{\sum_{r=1}^2X_n^{(r)}\big\Vert\log\Phi_{\cdot}\big(\frac{\theta_{n+1}}
              {\rho^{n+1}}\big)\big\Vert}E\Big(e^{\frac{\theta_{n+1}}{\rho^{n+1}}(\boldsymbol{u}\cdot I_1 )}\Big)\\
              \leq {} & e^{\tilde{u}^{-1}B_n\frac{\theta_{n+1}}{\theta_{n+1}/\rho^{n+1}}
              \big\Vert\log\Phi_{\cdot}\big(\frac{\theta_{n+1}}
              {\rho^{n+1}}\big)\big\Vert}E\Big(e^{\frac{\theta_{n+1}}{\rho^{n+1}}(\boldsymbol{u}\cdot I_1 )}\Big),
 \end{split}
\end{displaymath}
where $\tilde{u}=\min\{u_1,u_2\}$ and
$\big\Vert\log\Phi_{\cdot}(\frac{\theta_{n+1}}{\rho^{n+1}})\big\Vert
=\max\big\{\big|\log\Phi_r(\frac{\theta_{n+1}}{\rho^{n+1}})\big|, r=1,2\big\}.$
Then by (\ref{3.16}) and Lemma~\ref{le3.2}, there is a constant $\bar{C}_{11}$ such that
 \begin{equation}\label{3.17}
E_i\big(e^{\theta_{1}B_{1}}\big)=E_i\big(e^{\theta_{n+1}B_{n+1}}\big)\leq E_i\Big(e^{\theta_{n+1}B_n\big(1+\frac{\bar{C}_{11}}
{\rho^{n+1}}\big)}\Big)E\Big(e^{\frac{\theta_{n+1}}{\rho^{n+1}}(\boldsymbol{u}\cdot I_1 )}\Big).
\end{equation}
Iterating (\ref{3.17}) yields that
\begin{eqnarray*}
E_i\big(e^{\theta_{1}B_{1}}\big)
&\leq& E_i\bigg(e^{\theta_{n+1}\big(1+\frac{\bar{C}_{11}}
{\rho^{n+1}}\big)\cdots\big(1+\frac{\bar{C}_{11}}
{\rho^{2}}\big)B_1}\bigg)
E\bigg\{e^{\theta_{n+1}(\boldsymbol{u}\cdot I_1)\Big[\frac{1}{\rho^{n+1}}
+\frac{\big(1+\frac{\bar{C}_{11}}
{\rho^{n+1}}\big)}{\rho^n}+\cdots+\frac{\big(1+\frac{\bar{C}_{11}}
{\rho^{n+1}}\big)\cdots\big(1+\frac{\bar{C}_{11}}
{\rho^3}\big)}{\rho^2}\Big]}\bigg\}\\
&\leq& E_i\Big(e^{\theta_{n+1}B_1 C_{18}}\Big)
E\Big[e^{\theta_{n+1}(\boldsymbol{u}\cdot I_1)C_{18}\big(\frac{1}{\rho^{n+1}}+\cdots+\frac{1}{\rho^{2}}\big)}\Big],
\end{eqnarray*}
where $C_{18}=\prod\limits_{k=1}^{\infty}\big(1+\frac{\bar{C}_{11}}
{\rho^k}\big)<\infty$. Assume that $\lim\limits_{n\rightarrow \infty} \theta_n=\theta_0^\star=0$, then we see $E_i\big(e^{\theta_{1}B_{1}}\big)=1$, which is in contradiction to $\theta_{1}>0$, $B_1>0$.  \hfill $\Box$

From now on, we turn to prove Theorem~\ref{th3.5} by Theorem~\ref{th3.4}.
\par
\noindent
\text{\em{Proof of Theorem~\ref{th3.5}}}\ \ \ From (\ref{1.1}) and (\ref{2.16}), $Y_n=W_n+V_n,$
where $W_n=\rho^{-n}\left(\boldsymbol{u}\cdot Z_n\right)$ converges to $W$ w.p.1 as $n\rightarrow\infty$ (see \cite{AKN72}) and
$$
V_n=\frac{\boldsymbol{u}\cdot\big(U_n^{(1)}+U_n^{(2)}+\cdots+U_n^{(n)}\big)}{\rho^n}-\frac{\rho^{n+1}-1}
{\rho^n(\rho-1)}(\boldsymbol{u}\cdot \boldsymbol{\lambda}).
$$
Thus $V_n$ converges to a $r.v.$ $V^{(1)}$ w.p.1 as $n\rightarrow\infty$ by Proposition~\ref{pro2.3}. Notice that
 \begin{eqnarray*}
&& Y-Y_n\\
                & = & \lim_{k \to \infty}\bigg[\rho^{-(n+k)}\boldsymbol{u}\cdot X_{n+k}-\rho^{-n} \boldsymbol{u}\cdot X_n
                -\Big(\frac{\rho^{n+k+1}-1}{\rho^{n+k}(\rho-1)}
                -\frac{\rho^{n+1}-1}{\rho^n(\rho-1)}\Big)(\boldsymbol{u}\cdot \boldsymbol{\lambda})\bigg] \\
               &= &\lim_{k \to \infty}\rho^{-n}\bigg[\sum _{r=1}^2\sum_{m=1}^{X_n^{(r)}}\big(\rho^{-k} \boldsymbol{u}\cdot Z_{k,m,r}-u_r\big)
               +\frac{\boldsymbol{u}\cdot \sum_{i=n+1}^{n+k}U_{n+k}^{(i)}}{ \rho^{k}}
               -\frac{\rho^{k+1}-1}{\rho^k(\rho-1)}\big(\boldsymbol{u}\cdot \boldsymbol{\lambda}\big)\bigg] \\
               & & -\lim_{k \to \infty}\rho^{-n} \bigg[\Big(\frac{\rho^{n+k+1}-1}{\rho^k(\rho-1)}-\frac{\rho^{n+1}-1}{\rho-1}
               -\frac{\rho^{k+1}-1}{\rho^k(\rho-1)}\Big)(\boldsymbol{u}\cdot \boldsymbol{\lambda})\bigg]\\
               &= & \rho^{-n}\bigg(\sum _{r=1}^2\sum_{m=1}^{X_n^{(r)}}(W_m^{(r)}-u_r)+V+\boldsymbol{u}\cdot \boldsymbol{\lambda}\bigg),
\end{eqnarray*}
where $Z_{k,m,r}$ is the population vector of $k$th generation offspring of the $m$th type $r$ particle
among the $X_n^{(r)}$ particles existing in the $n$th generation. $W_m^{(r)}$ is the limit r.v.
in the line of descendant initiated by the $m$th type $r $ parent of the $n$th generation.
 Thus  $\{W_m^{(r)}\}_{m\geq1}$ are $i.i.d.$ for fixed $r$ and $W_1^{(1)}$ and $W_1^{(2)}$ are independent.
 $V$ is identically distributed as $V^{(1)}.$
 \par
 By conditional independence, for any $\varepsilon>0$,
$$
P(\left|Y-Y_n\right|>\varepsilon|\mathcal{F}_n)=\psi(X_n,\rho^n\varepsilon),
$$
where
$$
\psi(\boldsymbol{k},\rho^n\varepsilon)=P\bigg( \sum_{r=1}^{2}\sum_{m=1}^{k_r}T_m^{(r)}+V
+\boldsymbol{u}\cdot\boldsymbol{\lambda} \geq \rho^n\varepsilon \bigg)
$$
and $T_m^{(r)}=W_m^{(r)}-u_r, \boldsymbol{u}=(u_1, u_2), \boldsymbol{k}=(k_1,k_2).$
Therefore
$$
P(Y-Y_n>\varepsilon)=E(\psi(X_n,\rho^n\varepsilon)).
$$
\par
Next we give an estimate of $\psi(\boldsymbol{k},\varepsilon).$ For $0<\theta<\theta_0^{\star}$,
 using Markov's inequlity,
 \begin{displaymath}
 \begin{split}
 \psi(\emph{\textbf{k}},\varepsilon)
 ={} & P\bigg(\frac{\sum_{r=1}^{2}\sum_{j=1}^{k_r}T_m^{(r)}+V+\boldsymbol{u}\cdot \boldsymbol{\lambda} }
 {\sqrt{u\cdot\emph{\textbf{k}}}}>\frac{\varepsilon}{\sqrt{\boldsymbol{u}\cdot \emph{\textbf{k}}}}\bigg)\\
 \leq {} & \prod_{r=1}^2\Big(\varphi_r\Big(\frac{\theta} {\sqrt{\boldsymbol{u}\cdot \emph{\textbf{k}}}}\Big)e^{-\frac{\theta u_r}
  {\sqrt{u\cdot \emph{\textbf{k}}}}}\Big)^{k_r} \cdot E\left(e^{\frac{\theta}{\sqrt{\boldsymbol{u}\cdot \boldsymbol{k}}}
  \left(V+\boldsymbol{u}\cdot \boldsymbol{\lambda} \right)}\right)
   e^{-\frac{\theta\varepsilon} {\sqrt{\boldsymbol{u}\cdot \boldsymbol{k}}}},\\
 \end{split}
\end{displaymath}
where $\varphi_r(\theta)= E(e^{\theta W_1^{(r)}})<\infty$ for $\theta\leq \theta_0^{\star}$ since (\ref{3.15}).
 Observing that for $r=1,2$,
\begin{equation}\label{3.19}
\bigg(\varphi_r\Big(\frac{\theta} {\sqrt{\boldsymbol{u}\cdot \emph{\textbf{k}}}}\Big)e^{-\frac{\theta u_r}
{\sqrt{\boldsymbol{u}\cdot \emph{\textbf{k}}}}}\bigg)^{k_r}
=\bigg(1+\frac{1}{k_r} \frac{\varphi_r\big(\frac{\theta}{\sqrt{\boldsymbol{u}\cdot \emph{\textbf{k}}}}\big)
e^{-\frac{\theta u_r} {\sqrt{\boldsymbol{u}\cdot \emph{\textbf{k}}}}}-1}{\theta^2/k_r}\theta^2\bigg)^{k_r}\leq e^{C_{19}(r)}<\infty,
\end{equation}
where
\begin{eqnarray*}
\frac{1}{u_r}\sup\limits_{0\leq\theta\leq\theta^{\circ}}\frac{\Big|\varphi_r\left(\frac{\theta}
    {\sqrt{\boldsymbol{u}\cdot\emph{\textbf{k}}}}\right)e^{-\frac{\theta u_r}
    {\sqrt{\boldsymbol{u}\cdot\emph{\textbf{k}}}}}-1\Big|}{\theta^2/(u_rk_r)}
\leq   \frac{1}{u_r}\sup\limits_{0\leq\theta\leq\theta^{\circ}}\frac{\Big|\varphi_r\left(\frac{\theta}
    {\sqrt{\boldsymbol{u}\cdot\emph{\textbf{k}}}}\right)e^{-\frac{\theta u_r} {\sqrt{\boldsymbol{u}\cdot\emph{\textbf{k}}}}}-1\Big|}
    {\theta^2/(\boldsymbol{u}\cdot\emph{\textbf{k}})}
=   C_{19}(r)<\infty,
\end{eqnarray*}
since $$
\lim\limits_{\theta\rightarrow0}\big(\varphi_r(\theta)e^{-\theta u_r}-1\big)/\theta^2
=\frac{1}{2}\left(E\big((W_1^{(r)})^2\big)-2u_rE(W_1^{(r)})+u_r^2\right)<\infty
$$
and $\theta^{\circ}=min\{1,\theta_0^{\star}\}.$ However by (\ref{3.15}),
\begin{equation}\label{3.20}
C_{20}\equiv E\big(e^{\frac{\theta}{\sqrt{\boldsymbol{u}\cdot\emph{\textbf{k}}}}
(V+\boldsymbol{u}\cdot \boldsymbol{\lambda})}\big)<\infty \ \ \ for\ \theta\leq \theta^{\circ}.
\end{equation}
Combining (\ref{3.19}) and (\ref{3.20}), we get that
$
\psi(\emph{\textbf{k}},\varepsilon)\leq C_{21}e^{-\frac{\theta^{\circ}\varepsilon}
 {\sqrt{\boldsymbol{u}\cdot\emph{\textbf{k}}}}}$
and thus
 $$
 P(Y-Y_n\geq \varepsilon)=E(\psi(X_n,\rho^n\varepsilon))
 \leq C_{21}E\big( e^{-\frac{\theta^{\circ}\rho^n\varepsilon} {\sqrt{ \boldsymbol{u}\cdot X_n}}}\big)
 =C_{21}E\bigg(e^{- \frac{\theta^{\circ} \rho^{\frac {n}{2}}\varepsilon}
 {\sqrt{{Y_n+\frac{\rho^{n+1}-1}{\rho^n(\rho-1)}\boldsymbol{u}\cdot \boldsymbol{\lambda}}}}}\bigg).
 $$
For $\mu>0,$ by Theorem \ref{th3.4},
 \begin{displaymath}
 \begin{split}
 E\bigg(e^{-\frac{\mu }{\sqrt{{Y_n+\frac{\rho^{n+1}-1}{\rho^n(\rho-1)}\boldsymbol{u}\cdot \boldsymbol{\lambda}}}}}\bigg)
 ={} & \mu \int_0^{\infty}e^{-\mu x}P\bigg( \frac{1}{\sqrt{Y_n+\frac{\rho^{n+1}-1}
 {\rho^n(\rho-1)}\boldsymbol{u}\cdot \boldsymbol{\lambda}}} \leq x \bigg)dx\\
 \leq{} & \mu C_{22}\int_0^{\infty}e^{-\mu x}e^{\frac{-\theta_0^{\star}}{x^2}}dx \ \  \\
 ={} &C_{22} \int_0^{\infty}e^{-t}e^{-\frac{\theta_0^{\star}\mu^2}{t^2}}dt.
 \end{split}
\end{displaymath}
Thus
$$
P(Y-Y_n\geq \varepsilon)\leq C_{21}C_{22}\int_0^{\infty}e^{-t}e^{-\frac{\theta_0^{\star}\mu_n^2}{t^2}}dt,
$$
 where $\mu_n=\theta^{\circ}\varepsilon \rho^{\frac{n}{2}}.$ Indeed, for $\mu>0,$
 $I(\mu)=\int_0^{\infty}e^{-t}e^{-\frac{\mu^2}{t^2}}dt\leq 2e^{-\mu^{\frac{2}{3}}}.$
 Hence
$$
P(Y-Y_n\geq \varepsilon)\leq 2C_{21}C_{22}e^{-\big(\sqrt{\theta_0^{\star}}\theta^{\circ}\varepsilon
 \rho^{\frac{n}{2}}\big)^{\frac{2}{3}}}
=C_{4}e^{-\kappa_1 \big(\rho^{\frac{1}{3}}\big)^n\varepsilon^\frac{2}{3}},
$$
where $\kappa_1=(\sqrt{\theta_0^{\star}}\theta^{\circ})^{\frac{2}{3}}.$ Similar calculations hold for $P(Y-Y_n< \varepsilon).$
The proof is complete.\hfill $\Box$

The proof the Theorem~\ref{th3.6} will take up the rest of this section and actually that's how they proved in the non-immigration case~\cite{AV95} and the single-type case~\cite{Ath94,LL19}.
\par
\noindent
\text{\em{Proof of Theorem~\ref{th3.6}}}\ \ \ Observing the fact that for any $\varepsilon>0$, $\boldsymbol{l}\neq\textbf{0}$, $\alpha>0$, $0<\beta<1$,
\begin{eqnarray*}
P_i\left( \left| \frac{\boldsymbol{l}\cdot X_{n+1}}{\textbf{1}\cdot X_n}-\frac{\boldsymbol{l}\cdot (X_{n}M)}{\textbf{1}\cdot X_n}\right|
 >\varepsilon\Big|Y\geq \alpha\right)= p_{\alpha}(\alpha_{n1}+\alpha_{n2}),
\end{eqnarray*}
where
$$
\alpha_{n1}=P_i\left( \frac{|\boldsymbol{l}\cdot X_{n+1}-\boldsymbol{l}\cdot (X_{n}M)|}{\textbf{1}\cdot X_n}>\varepsilon,
 Y_n\leq \alpha\beta, Y\geq \alpha\right),$$ \ \
 $$\alpha_{n2}=P_i\left( \frac{|\boldsymbol{l}\cdot X_{n+1}-\boldsymbol{l}\cdot (X_{n}M)|}{\textbf{1}\cdot X_n}>\varepsilon,
 Y_n\geq \alpha\beta,Y\geq \alpha\right)
$$
and $p_\alpha=\frac{1}{P(Y\geq\alpha)}.$
Then, by (\ref{3.18}), there are finite constants $C_{23}$ and $\kappa_2$ such that
$$
\alpha_{n1}\leq P(Y-Y_n\geq \alpha(1-\beta))\leq C_{23}e^{-\kappa_2[\alpha(1-\beta)]^\frac{2}{3}\big(\rho^{\frac{1}{3}}\big)^n}.
$$
 Meanwhile, by (\ref{3.10d}) and Markov's inequality,
 \begin{displaymath}
 \begin{split}
 \alpha_{n2}
\leq{} & \sum_{\emph{\textbf{j}}} P\left( \frac{|\boldsymbol{l}\cdot X_{n+1}-\boldsymbol{l}\cdot (X_{n}M)|}{\textbf{1}\cdot X_n}>\varepsilon,Y_n\geq \alpha\beta\Big|X_n=\emph{\textbf{j}}\right)P_i(X_n=\emph{\textbf{j}})\\
 ={} & \sum_{\emph{\textbf{j}}} P\left( \frac{|\boldsymbol{l}\cdot X_{n+1}-\boldsymbol{l}\cdot (\emph{\textbf{j}}M)|}{\textbf{1}
 \cdot \emph{\textbf{j}}}>\varepsilon,\boldsymbol{u}\cdot \emph{\textbf{j}} \geq\alpha\beta\rho^n
 +\frac{\rho^{n+1}-1}{\rho-1}\boldsymbol{u}\cdot \boldsymbol{\lambda} \right) P_i(X_n=\emph{\textbf{j}})\\
 \leq {} & \sum_{\boldsymbol{u}\cdot \emph{\textbf{j}}\geq \alpha\beta\rho^n}P_i(X_n=\emph{\textbf{j}})
 P\left( \frac{\big|\sum_{r=1}^2 \sum_{m=1}^{j_r}\boldsymbol{l}\cdot(Z_{1,m,r}-e_rM)\big|}{\textbf{1}\cdot \emph{\textbf{j}}}>\frac{\varepsilon}{2}\right)
 +P\left( \frac{\left|\boldsymbol{l}\cdot I_1\right|}{\alpha\beta \rho^n}>\frac{\varepsilon}{2} \right)\\
 \leq {} & C_{24}e^{-\Lambda_1(\varepsilon)\alpha\beta \rho^n}+ E\left(e^{\theta_0\left(\textbf{1}\cdot I_1\right)}\right)e^{-\varepsilon(2\parallel l\parallel)^{-1}
{\theta_0}\alpha\beta\rho^n},\\
 \end{split}
\end{displaymath}
where the constants $C_{24}$ and $\Lambda_1(\varepsilon)$ is due to Chernoff type bounds
 since $ E\big(e^{\theta_0^{\star}(\textbf{1}\cdot X_1)}\big)<\infty$ for some $\theta_0^{\star}>0.$
Thus we can choose appropriate constants $C_{25}$ and $\Lambda_3(\varepsilon)$ such that
$$
\alpha_{n2}
\leq C_{25}e^{-\Lambda_3(\varepsilon)\alpha\beta \rho^n}.
$$
Hence there are constants $C_5$ and $C_6$ such that (\ref{3.21}) holds.
It remains to prove (\ref{3.22}). Note that for any $\varepsilon>0$, $\boldsymbol{l}\neq\textbf{0}$, $\alpha>0$, $0<\beta<1$,
$$
P_i\left(\frac{\boldsymbol{l}\cdot X_n}{\textbf{1}\cdot X_n}-\frac{\boldsymbol{l}\cdot \boldsymbol{v}}{\textbf{1}\cdot \boldsymbol{v} }>\varepsilon\Big|Y\geq \alpha\right)=p_{\alpha}(\alpha_{n3}+\alpha_{n4}),
$$
where
\begin{eqnarray*}
\alpha_{n3}=P_i\left( \frac{\boldsymbol{l}\cdot X_n}{\textbf{1}\cdot X_n}-\frac{\boldsymbol{l}\cdot \boldsymbol{v} }{\textbf{1}\cdot \boldsymbol{v} }>\varepsilon,
Y_{n-k_0}\leq \alpha\beta,Y\geq \alpha\right),
\end{eqnarray*}
\begin{eqnarray*}
\alpha_{n4}=P_i\left( \frac{\boldsymbol{l}\cdot X_n}{\textbf{1}\cdot X_n}-\frac{\boldsymbol{l}\cdot \boldsymbol{v} }{\textbf{1}\cdot \boldsymbol{v} }>\varepsilon,
Y_{n-k_0}\geq \alpha\beta,Y\geq \alpha\right)
\end{eqnarray*}
and $p_\alpha=\frac{1}{P(Y\geq\alpha)}.$
By (\ref{3.18}),
there are constants $C_{26}$ and $\kappa_3$ such that$$ \alpha_{n3}\leq P(Y-Y_{n-k_0}\geq \alpha(1-\beta))
\leq C_{26}e^{-\kappa_3[\alpha(1-\beta)]^\frac{2}{3}\big(\rho^\frac{1}{3}\big)^n}.$$
However, by (\ref{3.10a}) and Chernoff type bounds,
\begin{displaymath}
 \begin{split}
 \alpha_{n4}
\leq{} & \sum_{\emph{\textbf{j}}} P\left(\frac{\boldsymbol{l}\cdot X_n}{\textbf{1}\cdot X_n}-\frac{\boldsymbol{l}\cdot \boldsymbol{v} }{\textbf{1}\cdot \boldsymbol{v} }
>\varepsilon,Y_{n-k_0}\geq \alpha\beta\Big|X_{n-k_0}=\emph{\textbf{j}}\right)P_i(X_{n-k_0}=\emph{\textbf{j}})\\
 \leq {} & \sum_{\boldsymbol{u}\cdot\emph{\textbf{j}}\geq \alpha\beta\rho^{n-k_0}}P_i(X_{n-k_0}
 =\emph{\textbf{j}})P\left(\frac{\tilde{\boldsymbol{l}}\cdot\big(X_n-\emph{\textbf{j}}M^{k_0}
 -\sum_{i=0}^{k_0-1}\boldsymbol{\lambda} M^i\big)}{\boldsymbol{u}\cdot(\emph{\textbf{j}}+\frac{\boldsymbol{\lambda}}{\rho}
 +\cdots+\frac{\boldsymbol{\lambda}}
 {\rho^{k_0}})\rho^{k_0}}>\frac{\varepsilon(\textbf{1}\cdot \boldsymbol{v} )}{2}\right)\\
\leq{}&\sum_{\boldsymbol{u}\cdot\emph{\textbf{j}}\geq \alpha\beta\rho^{n-k_0}}P_i(X_{n-k_0}=\emph{\textbf{j}})
P\left(\frac{\tilde{\boldsymbol{l}}\cdot\big(X_n-\emph{\textbf{j}}M^{k_0}
-\sum_{i=0}^{k_0-1}\boldsymbol{\lambda} M^i\big)}{\alpha\beta\rho^n}>\frac{\varepsilon(\textbf{1}\cdot \boldsymbol{v} )}{2}\right)\\
 \leq {} & C_{27}e^{-\Lambda_4(\varepsilon)\alpha\beta \rho^n}.\\
 \end{split}
\end{displaymath}
  Similar calculations hold for the other side.
The proof is complete.\hfill $\Box$
\par

\section*{Acknowledgement}
\par
 We thank the referees for their time and comments.

\end{document}